%% file: MFSBP_main.tex
\documentclass[final,onefignum,onetabnum,letterpaper]{siamonline190516}
\usepackage[scale=0.8]{geometry} 

\input{MFSBP_shared}

\usepackage{lineno}

\ifpdf
\hypersetup{
  pdftitle={MFSBP operators},
  pdfauthor={Glaubitz et al.}
}
\fi

\begin{document}

\maketitle

\begin{abstract}
	\input{0_abstract}
\end{abstract}



\begin{DOI}
	\url{https://doi.org/10.1016/j.jcp.2023.112370}
\end{DOI}

\input{1_introduction} 
\input{2_notation}
\input{3_SBP_operators} 
\input{4_connection} 
\input{5_construction} 
\input{6_properties}  
\input{7_examples}
\input{8_num_test}

\input{9_summary}

\appendix 
\input{app_POCS}

\section*{Acknowledgements} 
This research was supported by AFOSR \#F9550-18-1-0316, the US DOD (ONR MURI) grant \#N00014-20-1-2595, the US DOE (SciDAC program) grant \#DE-SC0012704, Vetenskapsr\r{a}det Sweden grant 2018-05084 VR and 2021-05484, the Swedish e-Science Research Center (SeRC), and the Gutenberg Research College, JGU Mainz.
Furthermore, it was supported through the program ``Oberwolfach Research Fellows" by the Mathematisches Forschungsinstitut Oberwolfach in 2022. 
We also thank Maximilian Winkler for helpful discussions on the POCS algorithm.

\small
\bibliographystyle{siamplain}
\bibliography{references}

\end{document}

%% file: MFSBP_shared.tex
\usepackage{lipsum}
\usepackage{amsfonts}
\usepackage{epstopdf}
\ifpdf
  \DeclareGraphicsExtensions{.eps,.pdf,.png,.jpg}
\else
  \DeclareGraphicsExtensions{.eps}
\fi

\usepackage{datetime}
\newdateformat{monthyeardate}{%
  \monthname[\THEMONTH] \THEDAY, \THEYEAR}

\usepackage{academicons}
\usepackage{xcolor}
\renewcommand{\orcid}[1]{\href{https://orcid.org/#1}{\textcolor[HTML]{A6CE39}{orcid.org/#1}}}

\usepackage{amsmath}
\allowdisplaybreaks
\usepackage{amssymb}
\usepackage{commath}
\usepackage{mathtools}
\usepackage{bbm}

\usepackage{color}
\usepackage{graphicx}
\usepackage[small]{caption}
\usepackage{subcaption}

\usepackage{relsize}
\usepackage{adjustbox}
\usepackage{algorithm}
\usepackage[noend]{algpseudocode}
\usepackage{booktabs}
\usepackage{tikz}
\usetikzlibrary{shapes.geometric, arrows}
\tikzstyle{startstop} = [rectangle, rounded corners, minimum width=3cm, minimum height=1cm,text centered, draw=black]
\tikzstyle{process} = [rectangle, minimum width=3cm, minimum height=1cm, text centered, draw=black]

\tikzstyle{decision} = [diamond, minimum width=3cm, minimum height=1cm, text centered, draw=black]
\tikzstyle{arrow} = [thick,->,>=stealth]

\usepackage{verbatim}

\usepackage{enumitem}
\setlist[enumerate]{leftmargin=.5in}
\setlist[itemize]{leftmargin=.5in}

\newsiamthm{problem}{Problem}
\newsiamremark{remark}{Remark}
\newsiamremark{hypothesis}{Hypothesis} 
\crefname{hypothesis}{Hypothesis}{Hypotheses}
\newsiamremark{example}{Example}
\newsiamthm{claim}{Claim}
\newsiamthm{conjecture}{Conjecture}

\headers{MFSBP operators: Theory and construction}{J.\ Glaubitz, S.-C.\ Klein, J.\ Nordstr\"om, and P.\ \"Offner}

\title{Multi-dimensional summation-by-parts operators for general function spaces: \\ Theory and construction 
\thanks{
\monthyeardate\today 
\corresponding{Jan Glaubitz} 
}}

\author{
Jan Glaubitz\thanks{
Department of Aeronautics and Astronautics, Massachusetts Institute of Technology, USA 
(\email{glaubitz@mit.edu}, \orcid{0000-0002-3434-5563})
}
\and 
Simon-Christian Klein\thanks{
Department of Mathematics, TU Braunschweig, Germany (\email{simon-christian.klein@tu-bs.de}, \orcid{0000-0002-8710-9089})
}
\and 
{Jan Nordstr\"om\thanks{
Department of Mathematics, Link\"oping University, Sweden 
(\email{jan.nordstrom@liu.se}, \orcid{0000-0002-7972-6183})
}} 
\thanks{
Department of Mathematics and Applied Mathematics, University of Johannesburg, South Africa
}
\and 
Philipp \"Offner\thanks{
Institute of Mathematics, Johannes Gutenberg University Mainz, Germany 
(\email{poeffner@uni-mainz.de}, \orcid{0000-0002-1367-1917})
} 
}

\usepackage{amsopn}

\usepackage{soul}

\DeclareMathOperator{\diag}{diag}
 
\newcommand{\scp}[2]{\left\langle{#1, #2}\right\rangle}

\renewcommand{\d}{\mathrm{d}}
\newcommand{\intd}{\, \mathrm{d}}
\newcommand{\N}{\mathbb{N}}

\newcommand{\R}{\mathbb{R}} 

\newcommand{\fnum}{f^{\operatorname{num}}}

\newsavebox{\DelimiterBox}
\newlength{\DelimiterHeight}
\newlength{\DelimiterDepth}
\newsavebox{\ArgumentBox}
\newlength{\ArgumentHeight}
\newlength{\ArgumentDepth}
\newlength{\ResizedDelimiterHeight}
\newlength{\ResizedDelimiterDepth}

\ifluatex

\else

\fi

\newcommand{\vd}{\mathrm{d}}

\newcommand{\derd}[2]{\frac{\vd #1}{\vd #2}}

\DeclareMathOperator{\lspan}{span}

\newcommand{\be}{\begin{equation}}
\newcommand{\ee}{\end{equation}}
\renewcommand{\phi}{\varphi}
\renewcommand{\epsilon}{\varepsilon}

%% file: 0_abstract.tex
Summation-by-parts (SBP) operators allow us to systematically develop energy-stable and high-order accurate numerical methods for time-dependent differential equations. 
Until recently, the main idea behind existing SBP operators was that polynomials can accurately approximate the solution, and SBP operators should thus be exact for them. 
However, polynomials do not provide the best approximation for some problems, with other approximation spaces being more appropriate. 
We recently addressed this issue and developed a theory for \emph{one-dimensional} SBP operators based on general function spaces, coined function-space SBP (FSBP) operators. 
In this paper, we extend the theory of FSBP operators to \emph{multiple dimensions}. 
We focus on their existence, connection to quadratures, construction, and mimetic properties. 
A more exhaustive numerical demonstration of multi-dimensional FSBP (MFSBP) operators and their application will be provided in future works. 
Similar to the one-dimensional case, we demonstrate that most of the established results for polynomial-based multi-dimensional SBP (MSBP) operators carry over to the more general class of MFSBP operators. 
Our findings imply that the concept of SBP operators can be applied to a significantly larger class of methods than is currently done. 
This can increase the accuracy of the numerical solutions and/or provide stability to the methods. 

%% file: 1_introduction.tex
\section{Introduction} 
\label{sec:introduction} 

Summation-by-parts (SBP) operators mimic integration-by-parts on a discrete level. 
In combination with weakly enforced boundary conditions (BCs), they allow for a systematic development of energy-stable numerical methods for energy-bounded initial boundary value problems (IBVPs) \cite{svard2014review,fernandez2014review,chen2020review}. 
Examples include finite difference (FD) \cite{kreiss1974finite,kreiss1977existence,scherer1977energy,strand1994summation}, essentially non-oscillatory (ENO) and weighted ENO (WENO) \cite{yamaleev2009systematic,fisher2011boundary,carpenter2016entropy}, finite/spectral element (FE/SE) \cite{carpenter2014entropy,abgrall2020analysis,abgrall2021analysis}, discontinuous Galerkin (DG) \cite{gassner2013skew,chen2017entropy}, finite volume (FV) \cite{nordstrom2001finite,nordstrom2003finite}, flux reconstruction (FR) \cite{huynh2007flux,ranocha2016summation,offner2018stability}, and implicit time integration \cite{nordstrom2013summation,linders2020properties,ranocha2021new} methods.
At their core, existing SBP operators are constructed to be exact for polynomials up to a certain degree. 
The underlying assumption---although not always stated explicitly---is that polynomials accurately approximate the solution of the problem at hand. 
However, for some IBVPs, polynomials are not the best choice, and other approximation spaces should be used. 
Many previous works have pointed out the advantages of non-polynomial approximation spaces. 
These include exponentially fitted schemes to solve singular perturbation problems \cite{kadalbajoo2003exponentially,kalashnikova2009discontinuous}, DG methods \cite{yuan2006discontinuous} and (W)ENO reconstructions \cite{christofi1996study,iske1996structure,hesthaven2019entropy} based on non-polynomial approximation spaces, radial basis function (RBF) schemes \cite{fasshauer1996solving,fornberg2015solving,fornberg2015primer}, and methods based on rational function approximations \cite{nakatsukasa2018aaa,gopal2019solving}. 
In \cite{glaubitz2022summation}, we recently developed one-dimensional SBP operators for general (non-polynomial) function spaces, referred to as \emph{function-space SBP (FSBP) operators}. 
Although these can be applied to multi-dimensional problems using a tensor-product strategy, with the advantage of being simple and often efficient, this is not always the best choice since it limits their geometric flexibility. 

Here, we develop a theory for \emph{multi-dimensional} FSBP (MFSBP) operators. 
We focus on establishing their existence, connection to quadratures, construction, and mimetic properties of diagonal-norm MFSBP operators for general function spaces and geometries. 
It is demonstrated that most of the established results for previously developed multi-dimensional polynomial SBP (MSBP) operators \cite{nordstrom2003finite,hicken2016multidimensional,del2018simultaneous} carry over to MFSBP operators. 
We specifically connect the existence of MFSBP operators and positive quadratures that are exact for certain, in general, non-polynomial function spaces. 
Building upon the theoretical existence investigation, we derive a general construction procedure for MFSBP operators. 
An essential part of this procedure is that positive quadratures that are exact for a certain function space must be found if an $\mathcal{F}$-exact MFSBP operator approximating the partial derivative $\partial_x$ is desired. 
Note that the MFSBP operators presented here are not optimized w.r.t. to their grid. 
While we desire to investigate the optimal accuracy and efficiency of our MFSBP operators, such efforts need to be tailored to specific function spaces. 
They will therefore be carried out in future work. 

Our findings imply that the concept of MSBP operators can be applied to a significantly larger class of methods than is currently known. 
Another aspect of introducing the SBP framework in an existing ``old" method is to gain stability, as done for the FV method in \cite{nordstrom2001finite,nordstrom2003finite} and for RBF methods in \cite{glaubitz2022energy}, when combined with weakly enforced boundary data \cite{glaubitz2021stabilizing,glaubitz2021towards}. 
Here, we demonstrated the advantage of MFSBP operators for different linear problems.  
A more exhaustive numerical study will be provided in future work. 

The rest of this work is organized as follows. 
In \cref{sec:notation}, we establish the notation that will be used subsequently. 
In \cref{sec:SBP_operators}, we introduce the concept of MFSBP operators for general function spaces and geometries. 
\cref{sec:connection} addresses the relationship between MFSBP operators and certain surface and volume quadratures, characterizing the theoretical existence of diagonal-norm MFSBP operators. 
In \cref{sec:construction}, we describe how diagonal-norm MFSBP operators can be constructed. 
\Cref{sec:properties} demonstrates that the established mimetic properties of MSBP operators carry over to the larger class of MFSBP operators. 
In \cref{sec:examples}, we provide illustrative examples of MFSBP operators on triangles and circles. 
\Cref{sec:num_tests} offers simplistic numerical tests demonstrating the potential advantage of MFSBP operators. 
Finally, we offer concluding thoughts in \cref{sec:summary}. 

%% file: 2_notation.tex
\section{Notation} 
\label{sec:notation}

We use the following notation in the remainder of this work.   
Let $d \in \N$ be a positive integer, $\Omega \subset \R^d$ be an open and bounded domain with piecewise-smooth boundary $\partial \Omega$, and $\boldsymbol{n} = [n_{x_1},\dots,n_{x_d}]^T$ be the corresponding outward pointing unit normal. 
Moreover, let $S = \{ \mathbf{x}_n \}_{n=1}^N$ be a set of $N$ nodes on $\Omega$ and $\boldsymbol{\xi} \in \R^d$ be an arbitrary directional vector with length one. 
If $f \in C^1(\Omega)$ is a continuously differentiable function on $\Omega$, then 
\begin{equation} 
\begin{aligned}
	\mathbf{f} & = [ f(\mathbf{x}_1), \dots, f(\mathbf{x}_N) ]^T, \\ 
	\mathbf{f_{\boldsymbol{\xi}}} & = [ (\partial_{\boldsymbol{\xi}} f)(\mathbf{x}_1), \dots, (\partial_{\boldsymbol{\xi}} f)(\mathbf{x}_N) ]^T, 
\end{aligned}
\end{equation} 
denote the nodal values of $f$ and its directional derivative $\partial_{\boldsymbol{\xi}} f$ in $\boldsymbol{\xi}$-direction on the node set $S$, respectively. 
Here, $\partial_{\boldsymbol{\xi}} f$ is the directional derivative defined as $\partial_{\boldsymbol{\xi}} f = \nabla f \cdot \boldsymbol{\xi}$. 
Note that, for $i \in \{1,\dots,d\}$, the usual partial derivative $\partial_{x_i} f$ corresponds to the directional derivative $\partial_{\boldsymbol{\xi}} f$ in the canonical coordinate direction $\boldsymbol{\xi} = [0,\dots,0,1,0,\dots,0]$ with the number $1$ in the $i$th component. 

Let $\mathcal{P}_p(\R^d)$ be the space of the $d$-dimensional polynomials of total degree up to $p$, which has dimension $n_p^* = \binom{p+d}{d}$. 
Given an exponent vector $\boldsymbol{\alpha} = [\alpha_1,\dots,\alpha_d] \in \N^d$, recall that the total degree of a monomial $x_1^{\alpha_1} \cdots x_d^{\alpha_d}$ is the $\ell^1$-norm of $\boldsymbol{\alpha}$, $\| \boldsymbol{\alpha} \|_1 = \alpha_1 + \dots + \alpha_d$, and that the total degree of a polynomial is the largest total degree of all (non-zero coefficient) monomials spanning the polynomial. 
$\mathcal{P}_p(\R^d)$ is therefore spanned by the monomial basis functions, 
\begin{equation} 
	\boldsymbol{x^{\alpha}} := x_1^{\alpha_1} \cdots x_d^{\alpha_d}, 
	\quad \| \boldsymbol{\alpha} \|_1 \leq p,
\end{equation} 
where $\boldsymbol{\alpha} = [\alpha_1,\dots,\alpha_d]$ and $\boldsymbol{x} = [x_1,\dots,x_d]$.
The nodal values of the monomial basis functions $\boldsymbol{x}^{\boldsymbol{\alpha}}$ and their derivative in $\boldsymbol{\xi}$-direction, $\partial_{\boldsymbol{\xi}} \boldsymbol{x}^{\boldsymbol{\alpha}}$, on the node set $S$ are denoted by $\mathbf{x^{\boldsymbol{\alpha}}}$ and $\mathbf{(x^{\boldsymbol{\alpha}})_{\boldsymbol{\xi}}}$, respectively. 

%% file: 3_SBP_operators.tex
\section{Multi-dimensional SBP operators for general function spaces} 
\label{sec:SBP_operators} 

We start by describing how MSBP operators can be extended to general function spaces.

\subsection{Multi-dimensional polynomial SBP operators}
\label{sub:classic_SBP}

MSBP operators were first introduced for FV methods in \cite{nordstrom2003finite,svard2004stability,svard2006stable} and later for multi-block FD and SE methods in \cite{hicken2016multidimensional,del2018simultaneous}. 
The first derivative operator is defined as follows. 

\begin{definition}[SBP operators]\label{def:SBP} 
	$D_{\boldsymbol{\xi}} = P^{-1} Q_{\boldsymbol{\xi}}$ is an \emph{SBP operator of (total) degree $p$}, approximating the first derivative operator $\partial_{\boldsymbol{\xi}}$ on the node set $S$, if 
	\begin{enumerate}
		\item[(i)] 
		$D_{\boldsymbol{\xi}} \mathbf{x^{\boldsymbol{\alpha}}} = \mathbf{(x^{\boldsymbol{\alpha}})_{\boldsymbol{\xi}}}$ for $\| \boldsymbol{\alpha} \|_1 \leq p$,
		
		\item[(ii)]
		the norm matrix $P$ is symmetric and positive definite (SPD), 
		
		\item[(iii)] 
		$Q_{\boldsymbol{\xi}} + Q_{\boldsymbol{\xi}}^T = B_{\boldsymbol{\xi}}$, and 
		
		\item[(iv)] 
		the boundary matrix $B_{\boldsymbol{\xi}}$ satisfies 
		\begin{equation}\label{eq:boundary_matrix_poly}
			\left( \mathbf{x^{\boldsymbol{\alpha}}} \right)^T B_{\boldsymbol{\xi}} \mathbf{x^{\boldsymbol{\beta}}} = \oint_{\partial \Omega} \boldsymbol{x^{\alpha}} \boldsymbol{x^{\beta}} ( \boldsymbol{\xi} \cdot \boldsymbol{n} ) \intd s, \quad \| \boldsymbol{\alpha} \|_1, \| \boldsymbol{\beta} \|_1 \leq q,
		\end{equation}
		where $q \geq p$ and $( \boldsymbol{\xi} \cdot \boldsymbol{n} )$ is the inner product of the directional vector $\boldsymbol{\xi}$ and the outward pointing unit normal $\boldsymbol{n}$.
		
	\end{enumerate}
\end{definition}

Relation (i) ensures that $D_{\boldsymbol{\xi}}$ is an accurate approximation of the continuous operator $\partial_{\boldsymbol{\xi}}$ by requiring the operator to be exact for all $d$-dimensional polynomials of total degree up to $p$. 
Condition (ii) guarantees that $P$ induces a proper discrete inner product and norm, which are respectively given by $\scp{\mathbf{u}}{\mathbf{v}}_P = \mathbf{u}^T P \mathbf{v}$ and $\|\mathbf{u}\|^2_P = \mathbf{u}^T P \mathbf{u}$ for $\mathbf{u},\mathbf{v} \in \R^N$. 
Relation (iii) encodes the SBP property, which allows us to mimic integration-by-parts (IBP) on a discrete level. 
Recall that IBP for the $\boldsymbol{\xi}$-derivative reads 
\begin{equation}\label{eq:IBP}
	\int_{\Omega} u (\partial_{\boldsymbol{\xi}} v) \intd \boldsymbol{x} + \int_{\Omega} (\partial_{\boldsymbol{\xi}} u) v \intd \boldsymbol{x} 
		= \oint_{\partial \Omega} u v ( \boldsymbol{\xi} \cdot \boldsymbol{n} ) \intd s, \quad \forall u, v \in C^1(\Omega).
\end{equation}
The discrete version of \cref{eq:IBP}, which follows from (iii), is 
\begin{equation}\label{eq:SBP}
	\mathbf{u}^T P ( D_{\boldsymbol{\xi}} \mathbf{v} ) + ( D_{\boldsymbol{\xi}} \mathbf{u} )^T P \mathbf{v} = \mathbf{u}^T B_{\boldsymbol{\xi}} \mathbf{v}, \quad \forall \mathbf{u}, \mathbf{v} \in \R^N. 
\end{equation} 
Note that the two terms on the left-hand side of \cref{eq:SBP} approximate the related terms on the left-hand side of \cref{eq:IBP}.
Finally, (iv) in \cref{def:SBP} ensures that also the right-hand side of \cref{eq:SBP} accurately approximates the right-hand side of \cref{eq:IBP}. 
To this end, the boundary operator $B_{\boldsymbol{\xi}}$ must be exact for $d$-dimensional polynomials of total degree up to $p$.

\subsection{Multi-dimensional SBP operators for general function spaces} 
\label{sub:gen_SBP}

The main idea behind (i) in \cref{def:SBP} is that polynomials of total degree up to $p$ approximate the PDE solution well for $p$ high enough, and the differentiation operator $D_{\boldsymbol{\xi}}$ should thus be exact for them. 
That is, we can reformulate (i) in \cref{def:SBP} as 
\begin{equation}\label{eq:exactness_poly}
	D_{\boldsymbol{\xi}} \mathbf{f} = \mathbf{f_{\boldsymbol{\xi}}} \quad \forall f \in \mathcal{P}_p(\R^d),
\end{equation}
where $\mathcal{P}_p(\R^d)$ denotes the space of the $d$-dimensional polynomials of total degree up to $p$. 
It is now clearly possible to replace $\mathcal{P}_p(\R^d)$ with any other subspace of $C^1(\Omega)$. 
Suppose it is reasonable to approximate the solution $u$ using the function space $\mathcal{F} \subset C^1(\Omega)$. 
We then modify \cref{eq:exactness_poly} to 
\begin{equation}
	D_{\boldsymbol{\xi}} \mathbf{f} = \mathbf{f_{\boldsymbol{\xi}}} \quad \forall f \in \mathcal{F}.
\end{equation} 
Similarly, we can modify \cref{eq:boundary_matrix_poly} to 
\begin{equation}
	\mathbf{f}^T B_{\boldsymbol{\xi}} \mathbf{g} = \oint_{\partial \Omega} f g ( \boldsymbol{\xi} \cdot \boldsymbol{n} ) \intd s, \quad \forall f,g \in \mathcal{G},
\end{equation} 
with $\mathcal{F} \subset \mathcal{G}$, where for simplicity, we choose $\mathcal{G}= \mathcal{F}$. 
As for one-dimensional SBP operators \cite{glaubitz2022summation}, it is now natural to formulate the following generalization of multi-dimensional SBP operators, which we refer to as multi-dimensional function-space SBP (MFSBP) operators. 

\begin{definition}[MFSBP operators]\label{def:FSBP}
	$D_{\boldsymbol{\xi}} = P^{-1} Q_{\boldsymbol{\xi}}$ is an \emph{$\mathcal{F}$-based MFSBP operator}, approximating  the first derivative operator $\partial_{\boldsymbol{\xi}}$ on the node set $S$, if 
	\begin{enumerate}
		\item[(i)] 
		$D_{\boldsymbol{\xi}} \mathbf{f} = \mathbf{f_{\boldsymbol{\xi}}}$ for all $f \in \mathcal{F}$,
		
		\item[(ii)]
		the norm matrix $P$ is SPD, 
		
		\item[(iii)] 
		$Q_{\boldsymbol{\xi}} + Q_{\boldsymbol{\xi}}^T = B_{\boldsymbol{\xi}}$, and 
		
		\item[(iv)] 
		the boundary matrix $B_{\boldsymbol{\xi}}$ satisfies 
		\begin{equation}\label{eq:boundary_matrix}
			\mathbf{f}^T B_{\boldsymbol{\xi}} \mathbf{g} = \oint_{\partial \Omega} f g ( \boldsymbol{\xi} \cdot \boldsymbol{n} ) \intd s, \quad \forall f,g \in \mathcal{F}.
		\end{equation}
		
	\end{enumerate}
\end{definition} 

Note that only (i) and (iv) in \cref{def:FSBP} differ from \cref{def:SBP}. 
Consequently, most of the results for polynomial MSBP operators carry over to our MFSBP operators, as we will demonstrate in the remainder of this paper. 
A similar observation was made in the one-dimensional case \cite{glaubitz2022summation}. 

For simplicity, we focus on diagonal-norm MFSBP operators, for which the norm matrix $P$ is diagonal. 
This allows us to connect them to certain quadratures (see \cref{sec:connection}), which simplifies their construction (see \cref{sec:construction}) and analysis (see \cref{sec:properties}). 
Furthermore, diagonal-norm SBP operators enable certain splitting techniques \cite{nordstrom2006conservative,gassner2016split,ranocha2018stability} and the extension to variable coefficients including curvilinear coordinates \cite{svard2004coordinate,nordstrom2017conservation,ranocha2017extended,chan2019efficient}. 
For the same reasons, we also restrict the discussion to diagonal boundary matrices \cite{nordstrom2017conservation,del2018simultaneous,chen2020review}. 

%% file: 4_connection.tex
\section{MFSBP operators and quadratures} 
\label{sec:connection} 

We now investigate the connection between MFSBP operators and certain quadratures. 
Similar to the one-dimensional case \cite{glaubitz2022summation}, we show that an MFSBP operator exists if and only if a specific quadrature exists. 

\subsection{Volume quadratures} 
\label{sub:connection_CFs} 

We start by briefly commenting on multi-dimensional volume and surface quadratures \cite{engels1980numerical,cools1997constructing,davis2007methods}.
Consider the volume quadrature $I_{X,W}$ on $\Omega \subset \R^d$ with nodes $X = \{\mathbf{x}_n\}_{n=1}^N \subset \Omega$ and weights $W = \{w_n\}_{n=1}^N$, 
\begin{equation} 
	I_{X,W}[f] 
		:= \sum_{n=1}^N w_n f(\mathbf{x}_n) 
		\approx \int_{\Omega} f(\boldsymbol{x}) \intd \boldsymbol{x} 
		=: I[f],
\end{equation}
where $f: \Omega \to \R$ is a continuous function. 
We say that the quadrature $I_{X,W}$ is \emph{positive} if its weights are positive, i.\,e., if $w_n > 0$ for all $w_n \in W$. 
Furthermore, given a function space $\mathcal{G}$, we say that $I_{X,W}$ is \emph{$\mathcal{G}$-exact} if the \emph{exactness condition} 
\begin{equation}\label{eq:exactness_cond_volume}
	I_{X,W}[g] = I[g] \quad \forall g \in \mathcal{G} 
\end{equation} 
holds.

\subsection{Characterizing the existence of MFSBP operators} 
\label{sub:connection_char} 

As stated above, the restriction to diagonal norm matrices $P$ allows us to characterize the existence of diagonal-norm MFSBP operators in terms of positive quadratures. 
To this end, we introduce the concept of Vandermonde matrices. 
Let $\{ f_1,\dots,f_K \}$ be a basis of $\mathcal{F} \subset C^1(\Omega)$. 
Evaluating the basis functions at the nodes $S = \{\mathbf{x}_n\}_{n=1}^N$ and writing the corresponding function values as the columns of a matrix, we get the Vandermonde matrix 
\begin{equation}\label{eq:Vand}
	V = [\mathbf{f_1},\dots,\mathbf{f_K}] = 
	\begin{bmatrix} 
		f_1(\mathbf{x}_1) & \dots & f_K(\mathbf{x}_1) \\ 
		\vdots & & \vdots \\ 
		f_1(\mathbf{x}_N) & \dots & f_K(\mathbf{x}_N)
	\end{bmatrix}. 
\end{equation} 
Furthermore, for $\boldsymbol{\xi} \in \R^d$, we denote by 
\begin{equation}\label{eq:dx_F2}
	\partial_{\boldsymbol{\xi}} (\mathcal{F}^2) 
		= \{ \, \partial_{\boldsymbol{\xi}}(f g) \mid \, f, g \in \mathcal{F} \, \} 
		= \{ \, (\partial_{\boldsymbol{\xi}} f) g + f (\partial_{\boldsymbol{\xi}} g) \mid \, f, g \in \mathcal{F} \, \}
\end{equation}
the space of functions corresponding to the derivative in $\boldsymbol{\xi}$-direction of the product of two functions from $\mathcal{F} \subset C^1(\Omega)$. 
We are now positioned to formulate our main result on the connection between diagonal-norm MFSBP operators and positive and $\partial_{x_i}(\mathcal{F}^2)$-exact quadratures. 

\begin{theorem}\label{thm:connection}
	Let $\mathcal{F} \subset C^1(\Omega)$ and assume that the Vandermonde matrix $V$ in \cref{eq:Vand} has linearly independent columns. 
	Further, let $B_{\boldsymbol{\xi}}$ be a boundary matrix satisfying (iv) in \cref{def:FSBP}.
	Then there exists a diagonal-norm $\mathcal{F}$-based MFSBP operator $D_{\boldsymbol{\xi}} = P^{-1} Q_{\boldsymbol{\xi}}$ (with $Q_{\boldsymbol{\xi}} + Q_{\boldsymbol{\xi}}^T = B_{\boldsymbol{\xi}}$) on the nodes $S$ if and only if there exists a positive and $\partial_{\boldsymbol{\xi}}(\mathcal{F}^2)$-exact quadrature on $\Omega$ with nodes $S$. 
\end{theorem}

\cref{thm:connection} is well-known for polynomial MSBP operators. 
 The assertion is proved using the same arguments as in the proofs of \cite[Theorems 3.2 and 3.3]{hicken2016multidimensional} and \cite[Theorem 2]{del2018simultaneous}. 

\begin{remark} 
	The Vandermonde matrix $V$ in \cref{eq:Vand} is ensured to have linearly independent columns if the node set $S$ is $\mathcal{F}$-unisolvent.\footnote{We say that the node set $S \subset \Omega$ is \emph{$\mathcal{F}$-unisolvent}, where $\mathcal{F} \subset C(\Omega)$ is a linear function space if $f \in \mathcal{F}$ and $f(\mathbf{x}) = 0$ for all $\mathbf{x} \in S$ implies $f \equiv 0$.} 
	This requirement is not restrictive since we never encountered a Vandermonde matrix $V$ with linearly dependent columns for $N \geq K$ in our numerical tests. 
	This might not be surprising since $V$ having linearly independent columns can be ensured if sufficiently many dense nodes are used \cite[Section 2.1]{glaubitz2022constructing}. 
\end{remark}

\subsection{Surface quadratures}
\label{sub:surface_quad}

We can also connect the existence of the boundary operator $B_{\boldsymbol{\xi}}$ in \cref{def:FSBP} to certain surface quadratures on the boundary $\Gamma = \partial \Omega$. 
Consider the surface quadrature $I^{\Gamma}_{Y,V}$ on the closed boundary $\Gamma = \partial \Omega$ with surface nodes $Y = \{ \mathbf{x}_m \}_{m=1}^M \subset \Gamma$ and weights $V = \{v_m\}_{m=1}^M$ satisfying 
\begin{equation}\label{eq:surface_QF}
	I_{Y,V}^{(\Gamma)}[f] 
		:= \sum_{m=1}^M v_m ( \boldsymbol{\xi} \cdot \boldsymbol{n}(\mathbf{x}_m) ) f(\mathbf{x}_m) 
		\approx \oint_{\Gamma} f ( \boldsymbol{\xi} \cdot \boldsymbol{n} ) \intd s 
		=: I^{(\Gamma)}[f],
\end{equation} 
where $f: \Omega \to \R$ again is a continuous function. 
We say that the surface quadrature $I_{Y,V}^{(\Gamma)}$ in \cref{eq:surface_QF} is \emph{positive} if its weights $v_1,\dots,v_M$ are positive. 
Furthermore, given a function space $\mathcal{G} \subset C(\Omega)$, we say that $I_{Y,V}^{(\Gamma)}$ is \emph{$\mathcal{G}$-exact} if the exactness condition 
\begin{equation}\label{eq:exactness_cond_surface}
	I_{Y,V}^{(\Gamma)}[g] = I^{(\Gamma)}[g] \quad \forall g \in \mathcal{G} 
\end{equation} 
holds. 
We pre-empt that the function space $\mathcal{G}$ for which we need exact surface quadratures will be $\mathcal{F}^2$ for reasons explained in the following. 

\begin{remark} 
	The relation \cref{eq:surface_QF} assumes that we have access to the exact normals. 
	This does not introduce a constraint for domains with a piecewise linear boundary, where the normals remain constant on each linear boundary segment.  
	However, for complicated geometries, the normals include higher-order terms that must be integrated alongside $\mathcal{F}^2$ by the surface quadrature. 
	This integration may necessitate a larger number of surface points. 
	Importantly, the added complexity is not exclusive to the proposed MFSBP operators but is generally valid.
\end{remark}

\subsection{Characterizing the existence of boundary operators} 
\label{sub:connection_boundaryM} 

We now connect the boundary operators $B_{\boldsymbol{\xi}}$, which are crucial to \cref{def:FSBP} of MFSBP operators $D_{\boldsymbol{\xi}} = P^{-1} Q_{\boldsymbol{\xi}}$ with $Q_{\boldsymbol{\xi}} + Q_{\boldsymbol{\xi}}^T = B_{\boldsymbol{\xi}}$, to surface quadratures on the boundary $\Gamma = \partial \Omega$. 
We restrict the discussion to diagonal boundary operators supported on the surface nodes. 
 
\begin{definition}[Mimetic boundary operators]\label{def:mimeticB} 
	Let $B_{\boldsymbol{\xi}}  \in \R^{N \times N}$ be a boundary operator on $\Gamma = \partial \Omega$ with 
	\begin{equation}\label{eq:boundary_matrix2}
		\mathbf{f}^T B_{\boldsymbol{\xi}} \mathbf{g} = \oint_{\Gamma} f g ( \boldsymbol{\xi} \cdot \boldsymbol{n} ) \intd s, \quad \forall f,g \in \mathcal{F}, 
	\end{equation} 
	where $( \boldsymbol{\xi} \cdot \boldsymbol{n} )$ is the inner product of the directional vector $\boldsymbol{\xi}$ and the outward pointing unit normal $\boldsymbol{n}$. 
	Assume that $B_{\boldsymbol{\xi}}$ is defined on the node set $S = \{\mathbf{x}_n\}_{n=1}^N \subset \Omega$. 
	We call $B_{\boldsymbol{\xi}}$ \emph{mimetic} if  
	\begin{enumerate}
	
		\item[(i)] 
		$B_{\boldsymbol{\xi}}$ is diagonal, i.e., $B_{\boldsymbol{\xi}} = \diag(b_1,\dots,b_N)$;
	
		\item[(ii)] 
		 $B_{\boldsymbol{\xi}}$ only acts on the boundary nodes, i.e., $b_n = 0$ if $\mathbf{x}_n \not\in \Gamma$ for $n=1,\dots,N$;
		
		\item[(iii)] 
		The nonzero entries of $B_{\boldsymbol{\xi}}$ are of the form $b_n = v_n ( \boldsymbol{\xi} \cdot \boldsymbol{n}(\mathbf{x}_n) )$, where $v_n$ is a positive weight.  
		
	\end{enumerate}
\end{definition} 

\begin{remark} 
	The mimetic boundary operators in \cref{def:mimeticB} incorporate all geometric and direction information necessary for performing the IBP procedure, which makes them closely connected to the encapsulated boundary operators used in	\cite{lundquist2018hybrid,aalund2019encapsulated,lundquist2022multi}. 
\end{remark}

We will see in a moment that \cref{def:mimeticB} allows us to identify the boundary matrix $B_{\boldsymbol{\xi}}$ with a positive and $\mathcal{F}^2$-exact surface quadrature. 

\begin{theorem}\label{thm:mimetic_B}
	Let $\mathcal{F} \subset C(\Omega)$, $\Omega$ be an open and bounded domain with closed boundary $\Gamma = \partial \Omega$, and $S = \{ \mathbf{x}_n \}_{n=1}^N$ be a node set on $\Omega$. 
	There exists a mimetic boundary operator $B_{\boldsymbol{\xi}}$ satisfying \cref{eq:boundary_matrix2} if and only if there exists a positive and $\mathcal{F}^2$-exact surface quadrature $I^{(\Gamma)}_{Y,V}$ with points $Y \subset S \cap \Gamma$. 
	Moreover, the non-zero entries of $B_{\boldsymbol{\xi}}$ correspond to the products of the surface quadrature weights with $\boldsymbol{\xi} \cdot \boldsymbol{n}$ at the corresponding surface point. 
\end{theorem} 

\begin{proof} 
	Denote the surface quadrature points by $Y = \{ \mathbf{y}_m \}_{m=1}^M$ and the weights by $V = \{ v_m \}_{m=1}^M$. 
	The proof consists of two parts. 
	We first show that the existence of a positive and $\mathcal{F}^2$-exact surface quadrature implies the existence of a mimetic boundary operator. 
	To this end, we construct as a diagonal matrix $B_{\boldsymbol{\xi}} = \diag( b_1, \dots, b_N )$ as 
	\begin{equation} 
		b_n = 
		\begin{cases} 
			v_m ( \boldsymbol{\xi} \cdot \boldsymbol{n}(\mathbf{x}_n) ) & \text{if } \mathbf{x}_n = \mathbf{y}_m, \\ 
			0 & \text{otherwise.} 
		\end{cases} 
	\end{equation} 
	By construction, $B_{\boldsymbol{\xi}}$ satisfies (i), (ii), and (iii) in \cref{def:mimeticB}. 
	It remains to show that \cref{eq:boundary_matrix2} holds. 
	To this end, we can rewrite the left-hand side of \cref{eq:boundary_matrix2} as 
	\begin{equation}\label{eq:boundary_matrix_proof1} 
		\mathbf{f}^T B_{\boldsymbol{\xi}} \mathbf{g} 
			= \sum_{m=1}^M v_m ( \boldsymbol{\xi} \cdot \boldsymbol{n}(\mathbf{y}_m) ) f(\mathbf{y}_m) g(\mathbf{y}_m) 
			= I^{(\Gamma)}_{Y,V}[fg], 
	\end{equation}
	using our above construction of $B_{\boldsymbol{\xi}}$. 
	Now let $f,g \in \mathcal{F}$. 
	Then $fg \in \mathcal{F}^2$, and since $I^{(\Gamma)}_{Y,V}$ is $\mathcal{F}^2$-exact, we have $I^{(\Gamma)}_{Y,V}[fg] = I^{(\Gamma)}[fg]$. 
	Substituting this into \cref{eq:boundary_matrix_proof1} yields 
	\begin{equation}
		\mathbf{f}^T B_{\boldsymbol{\xi}} \mathbf{g} 
			= I^{(\Gamma)}[fg] 
			= \oint_{\Gamma} f g ( \boldsymbol{\xi} \cdot \boldsymbol{n} ) \intd s, 
	\end{equation} 
	which shows that $B_{\boldsymbol{\xi}}$ satisfies \cref{eq:boundary_matrix2} and therefore is a mimetic boundary operator. 
	In the second part of the proof, similar arguments as in the first part can be used to show that the existence of a mimetic boundary operator $B_{\boldsymbol{\xi}} = \diag( b_1, \dots, b_N )$ implies the existence of a positive and $\mathcal{F}^2$-exact surface quadrature $I^{(\Gamma)}_{Y,V}$. 
\end{proof}

%% file: 5_construction.tex
\section{Construction of MFSBP operators} 
\label{sec:construction} 

Let $\Omega \subset \R^d$ be an open and bounded reference domain, let $\boldsymbol{\xi} \in \R^d$, and let ${\mathcal{F} \subset C^1(\Omega)}$ be the function space for which we want a diagonal-norm MFSBP operator $D_{\boldsymbol{\xi}}$ approximating $\partial_{\boldsymbol{\xi}}$. 
The construction then proceeds as follows: 
\begin{enumerate} 

	\item[(S1)] 
	Select $N_s$ and $N_i$ nodes on the surface and interior of $\Omega$, respectively. 
	The node set $S$ is the union of the surface and interior nodes and contains a total number of $N = N_s + N_i$ nodes. 

	\item[(S2)] 
	Find a positive and $\mathcal{F}^2$-exact surface quadrature supported on the $N_s$ surface points.  
	Use the surface quadrature to construct a mimetic boundary matrix $B_{\boldsymbol{\xi}}$. 

	\item[(S3)]  
	Find a positive and $\partial_{\boldsymbol{\xi}} ( \mathcal{F}^2 )$-exact volume quadrature supported on the $N$ surface and interior points.  
	Use the volume quadrature to construct a diagonal-norm matrix $P$. 
	
	\item[(S4)]  
	Find a matrix $Q_{\boldsymbol{\xi}}$ that satisfies \cref{def:FSBP}. 
	Then construct an MFSBP operator as ${D_{\boldsymbol{\xi}} = P^{-1} Q_{\boldsymbol{\xi}}}$. 
	
\end{enumerate}

We address details on (S2), (S3), and (S4) in \cref{sub:constr_B,sub:constr_P,sub:constr_Q}.
Furthermore, we comment on (S1), particularly how the nodes can be distributed, in \cref{sub:constr_nodes}.

\subsection{The boundary matrix $B_{\boldsymbol{\xi}}$}
\label{sub:constr_B}

We first describe how one can construct a mimetic boundary matrix $B_x$, satisfying \cref{def:mimeticB}. 
Given is a reference element $\Omega$ with piecewise smooth boundary $\Gamma$, where we denote the smooth parts by $\Gamma_1,\dots,\Gamma_J$. 
If we find mimetic boundary matrices $B_j$ on $\Gamma_j$, $j=1,\dots,J$, with
\begin{equation}\label{eq:constr_surface_quad1}
	\mathbf{f}^T B_j \mathbf{g} 
		= \int_{\Gamma_j} f g ( \boldsymbol{\xi} \cdot \boldsymbol{n} ) \intd s, \quad \forall f,g \in \mathcal{F},
\end{equation}
then we get the desired mimetic boundary matrix $B_{\boldsymbol{\xi}}$ as $B_{\boldsymbol{\xi}} = \sum_{j=1}^J B_j$.
Following \cref{thm:mimetic_B}, we can find such matrices $B_j$ in two steps: 
(i) Construct a positive and $\mathcal{F}^2$-exact surface quadrature, 
\begin{equation}
	I_{X^{(j)},V^{(j)}}^{(\Gamma_j)}[fg] 
		= \sum_{m=1}^{M_j} v^{(j)}_m ( \boldsymbol{\xi} \cdot \boldsymbol{n}(\mathbf{x}^{(j)}_m) ) f(\mathbf{x}^{(j)}_m) g(\mathbf{x}^{(j)}_m)
		\approx \oint_{\Gamma_j} f g ( \boldsymbol{\xi} \cdot \boldsymbol{n} ) \intd s 
		= I^{(\Gamma_j)}[fg].
\end{equation}
(ii) Use the products $v^{(j)}_m ( \boldsymbol{\xi} \cdot \boldsymbol{n}(\mathbf{x}^{(j)}_m) )$ as the non-zero diagonal entries of $B_j$ that correspond to the surface nodes on $\Gamma_j$. 
That is, we get $B_j = \diag( (b_j)_1, \dots, (b_j)_N )$ as 
\begin{equation} 
	(b_j)_n = 
	\begin{cases} 
		v^{(j)}_m ( \boldsymbol{\xi} \cdot \boldsymbol{n}(\mathbf{x}^{(j)}_m) ) & \text{if } \mathbf{x}_n = \mathbf{x}^{(j)}_m \in \Gamma_j, \\ 
		0 & \text{otherwise.} 
	\end{cases} 
\end{equation} 
We were able to find positive and $\mathcal{F}^2$-exact surface quadratures using equidistant points on $\Gamma_j$ using the least-squares approach \cite{glaubitz2020stableDG,glaubitz2020stableQF,glaubitz2022constructing} and the Projection Onto Convex Sets (POCS) algorithm \cite{Neumann1950Functional,GUBIN1967The,escalante2011alternating} if sufficiently many grid points were used. 
Given the two overlapping closed convex sets $C$ and $D$, the POCS algorithm finds a point $x \in C \cap D$ by alternatingly projecting onto the sets $C$ and $D$. 
To find a positive and $\mathcal{F}^2$-exact surface quadrature, the sets $C$ and $D$ encode the positivity of the weights and exactness conditions of an $\mathcal{F}^2$-exact surface quadrature, respectively. 
\cref{fig:POCS} illustrates the POCS algorithm for the case of $C$ and $D$ being (affine) linear spaces. 
We chose POCS because it is remarkably versatile. 
POCS can be employed not only for solving the surface quadrature but also for the volume quadrature (necessitating the enforcement of a lower bound constraint) and the anti-symmetric matrix $Q_A$ (requiring an anti-symmetric constraint). 
Switching between these diverse constraints demands minimal alterations to the POCS algorithm and its implementation. 
Concurrently, we noted that the computational costs of the POCS algorithm were competitive when compared to alternative methods. 
See \cref{app:POCS_surface} for more details on the POCS algorithm. 

\begin{figure}[tb]
	\centering
	\includegraphics[width=0.5\textwidth]{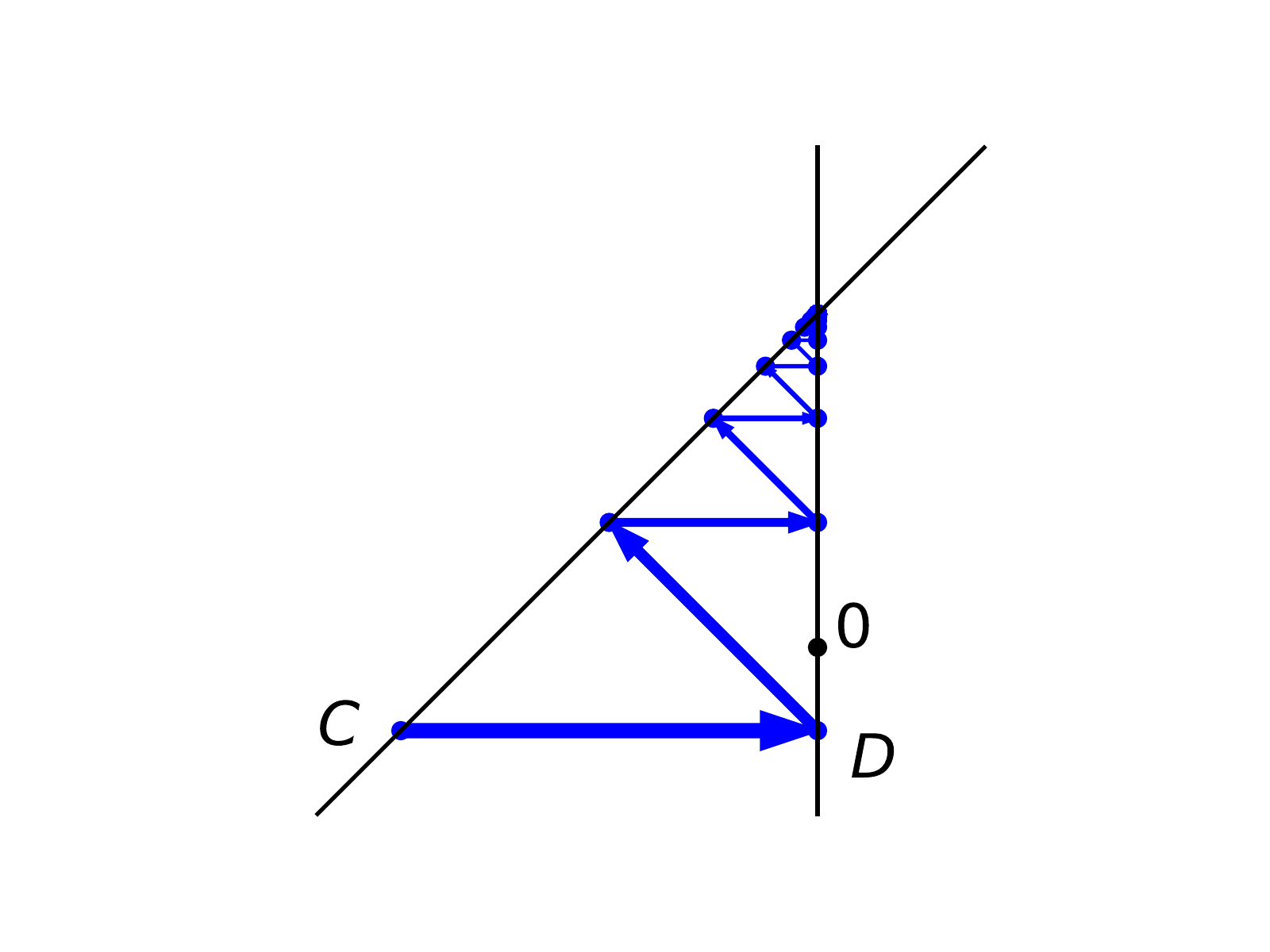} 
	\caption{ 
	Illustration of the POCS algorithm. 
	The two convex sets $C$, and $D$ correspond to (affine) linear subspaces. 
	 Starting from an initial point, the POCS algorithm produces a sequence of points (illustrated by the blue dots) that converges to a point in the intersection of $C$ and $D$. 
	}
	\label{fig:POCS}
\end{figure}

\begin{remark}\label{rem:corners}
	It is a non-trivial task in hyperbolic problems to properly treat element corners, where the normal vector becomes undefined. 
	To avoid that difficulty (which exists for all numerical methods), we do not place any points on the corners of the domain $\Omega$. 
\end{remark}

\begin{remark}\label{rem:conservation} 
	Inter-element conservation in a multi-element/block method can most easily be ensured if the quadratures at the surface that connect two elements (i) have their points at the same locations and (ii) use the same weights $v_m^{(j)}$ at these points. 
	Then, the inter-element contributions of both elements cancel out, and mass is neither created nor destroyed between elements. 
	Otherwise, special projection and coupling methods have to be utilized \cite{ranocha2017extended,chan2018discretely,chan2019skew}. 
\end{remark}

\begin{remark}[Computational costs]
	Most of the computational expense in finding the boundary matrix is dedicated to the numerical computation of the basis functions' moments (their integrals) up to machine precision, which is needed for the exactness conditions. 
	Relative to this, executing the POCS algorithm is an inexpensive process. 
	When executed on a single core, the total time to construct any of the two-dimensional MFSBP operators considered in the later numerical tests (see \cref{sec:examples} for more details) on the reference element was on the order of a second. 
\end{remark}

\subsection{The diagonal norm matrix $P$}
\label{sub:constr_P}

Assume that we have a positive and $\partial_{\boldsymbol{\xi}} (\mathcal{F}^2)$-exact volume quadrature $I_{X,W}$ on $\Omega$ with nodes $X$ and weights $W = \{w_n\}_{n=1}^N$. 
Following \cref{thm:connection}, we can construct a diagonal-norm MFSBP operator $D_{\boldsymbol{\xi}} = P^{-1} Q_{\boldsymbol{\xi}}$ on the nodes $X$ with a diagonal norm matrix 
\begin{equation}
	P = \diag(w_1,\dots,w_N). 
\end{equation} 
In general, we found such a positive and $\partial_{\boldsymbol{\xi}} (\mathcal{F}^2)$-exact volume quadrature using the $N_s$ surface points from constructing the mimetic boundary matrix $B_{x_i}$ (see \cref{sub:constr_B}), adding $N_i$ points in the interior of $\Omega$, and then using the POCS algorithm. 
See \cref{app:POCS_volume} for more details. 

\begin{remark} 
	Consider a discretized problem that necessitates using MFSBP operators in different directions. 
	For instance, in the $x$-direction, we have $D_x = P^{-1} Q_x$, while in the $y$-direction, we have $D_y = P^{-1} Q_y$. 
	The norm operator $P$ must be the same for both MFSBP operators to ensure energy stability, which we demonstrate in \cref{sub:linear_energy}.
	In such scenarios, $P$ must correspond to a volume quadrature that is positive and concurrently exact for both $\partial_x (\mathcal{F}^2)$ and $\partial_y (\mathcal{F}^2)$.
	The two-dimensional function spaces that we examine later in \cref{sec:examples,sec:num_tests} satisfy the condition $\partial_x (\mathcal{F}^2) = \partial_y (\mathcal{F}^2)$. 
	This congruence significantly simplifies the construction procedure. 
	On the other hand, if $\partial_x \mathcal{F}^2 \neq \partial_y \mathcal{F}^2$, then the union of $\partial_x \mathcal{F}^2$ and $\partial_y \mathcal{F}^2$ encompasses a larger function space than each individual one and a larger number of quadrature points might be necessary.
\end{remark}

\subsection{The matrix $Q_{\boldsymbol{\xi}}$ and the MFSBP operator $D_{\boldsymbol{\xi}} = P^{-1} Q_{\boldsymbol{\xi}}$}
\label{sub:constr_Q}

In the last step, we construct the matrix $Q_{\boldsymbol{\xi}}$ from which we get the desired MFSBP operator as $D_{\boldsymbol{\xi}} = P^{-1} Q_{\boldsymbol{\xi}}$. 
To this end, $Q_{\boldsymbol{\xi}}$ is decomposed into its symmetric and anti-symmetric parts, $Q_S$ and $Q_A$. 
Condition (iii) in \cref{def:FSBP} then yields 
\begin{equation}\label{eq:Q_split}  
	Q_{\boldsymbol{\xi}} = Q_A + \frac{1}{2} B_{\boldsymbol{\xi}}.
\end{equation} 
Furthermore, the accuracy condition (i) in \cref{def:FSBP} implies 
\begin{equation}\label{eq:QA_exactness} 
	Q_A V = P V_{\boldsymbol{\xi}} - \frac{1}{2} B_{\boldsymbol{\xi}} V.
\end{equation}
Here, $V_{\boldsymbol{\xi}} = [\mathbf{(f_1)_{\boldsymbol{\xi}}},\dots,\mathbf{(f_K)_{\boldsymbol{\xi}}}]$ is the Vandermonde matrix for the derivatives in $\boldsymbol{\xi}$-direction of the basis functions $f_1,\dots,f_K$ of the function space $\mathcal{F} \subset C^1(\Omega)$. 
It remains to construct an anti-symmetric matrix $Q_A$ that satisfies \cref{eq:QA_exactness}.  
Once such $Q_A$ is found, we get $Q$ by \cref{eq:Q_split}. 
Finding an anti-symmetric solution $Q_A$ of \cref{eq:QA_exactness} can be formulated as a classical Procrustes problem \cite{higham1988symmetric,andersson1997constrained,gower2004procrustes}, with many available solution procedures. 
For instance,  \cite{hicken2016multidimensional,del2018simultaneous,glaubitz2022summation} used a least-squares approach to determine $Q_A$. 
In our implementation, we used the POCS algorithm to find an anti-symmetric $Q_A$ that satisfies \cref{eq:QA_exactness} for efficiency reasons. 
We refer to \cref{app:POCS_QA} for more details on how the POCS algorithm can be used to find $Q_A$.

\subsection{The surface and interior nodes}
\label{sub:constr_nodes} 

The only restriction on the surface points is that a positive and $\mathcal{F}^2$-exact surface quadrature has to exist for them, while the only limitation on the inner points is that a positive and $\partial_{\boldsymbol{\xi}} (\mathcal{F}^2)$-exact volume quadrature has to exist for the union of the inner and surface points. 
To this end, it was proved in \cite{glaubitz2021stableCFs,glaubitz2022constructing} that a positive and $\mathcal{G}$-exact quadrature can be found whenever (i) $\mathcal{G}$ includes constants and (ii) sufficiently many points are used, where $\mathcal{G}$ is $\mathcal{F}^2$ and $\partial_{\boldsymbol{\xi}} (\mathcal{F}^2)$ in the case of the surface and volume quadrature, respectively. 
It was numerically observed in \cite{glaubitz2020stableQF,glaubitz2021stableCFs,glaubitz2022constructing} that the number of points has to be proportional to the squared dimension of $\mathcal{G}$. 
Finally, the proof presented in \cite{glaubitz2021stableCFs,glaubitz2022constructing} allows for any class of equidistributed points \cite{kuipers2012uniform}, including equidistant points and low discrepancy points used in quasi-Monte Carlo methods \cite{caflisch1998monte,dick2013high}, e.g., Halton points \cite{halton1960efficiency}. 
In our implementation, we used equidistant surface points and inner points that are Halton-distributed. 
More precisely, we started by using a single surface point on each piecewise smooth boundary part and successively increased this number by one until we found a positive and $\mathcal{F}^2$-exact surface quadrature. 
Afterward, starting from using only the surface points (no inner points), we successively increased the number of inner (Halton) points by one until we found a positive and $\partial_{\boldsymbol{\xi}} \mathcal{F}^2$-exact volume quadrature. 

\begin{remark}[Halton points]
	The main idea behind Halton points is to generate uniformly distributed and well-spaced points throughout the unit cube. 
	This is achieved by using different prime base numbers for each dimension and creating a sequence of numbers in each dimension that avoids large gaps or clusters. 
	To generate Halton points within a bounded domain $\Omega$, we initially create these points within a cube that encloses $\Omega$, subsequently utilizing only the subset of points that lie on $\Omega$. 
	For a more comprehensive understanding of Halton points, see \cite{halton1960efficiency,kuipers2012uniform}.
	We stress that using Halton points, while beneficial, is not a prerequisite for our method of constructing MFSBP operators. 
\end{remark}

%% file: 6_properties.tex
\section{Mimetic properties of MFSBP operators} 
\label{sec:properties}

Here, we demonstrate that (i) it is necessary to include constants in the function space $\mathcal{F}$ for conservation and that (ii) most other mimetic results for polynomial MSBP operators \cite{hicken2016multidimensional,del2018simultaneous} also hold for our MFSBP operators. 
For simplicity, we illustrate this for the scalar linear advection equation in two dimensions.

\subsection{The linear advection equation}
\label{sub:linear}

The initial boundary value problem (IBVP) for the scalar linear advection equation with constant coefficients is   
\begin{equation}\label{eq:linear}
\begin{aligned}
	\partial_t u+ a \partial_x u+ b \partial_y u & = 0, \quad && \forall (x,y) \in \Omega, \ t>0, \\ 
	u(0,x,y) & = u_0(x,y), \quad && \forall (x,y) \in \Omega, \\
	u(t,x,y) & = g(t,x,y), \quad && \forall (x,y) \in \Gamma_{-}, \ t \geq 0, 
\end{aligned}
\end{equation}
where $\Gamma_{-} = \{ \, (x,y) \in \partial\Omega \mid a n_x+ b n_y<0 \, \}$ is the inflow part of the boundary of $\Omega$ and $\Gamma_{+} = \Gamma \setminus \Gamma_{-}$ is the outflow part. 
Here, $n_x$ and $n_y$ are the first and second component of the outward pointing unit normal $\boldsymbol{n} = [n_x,n_y]^T$ for the boundary $\Gamma$. 
Let $D_x=P^{-1} Q_x$ and $D_y=P^{-1} Q_y$ be MFSBP operators, approximating the partial derivatives in the canonical directions, $\partial_x$ and $\partial_y$, with mimetic boundary operators $B_x$ and $B_y$. 
The MFSBP-SAT semi-discretization of \cref{eq:linear} is formally given by 
\begin{equation}\label{eq:linear_discr}
	\mathbf{u}_t + a D_x \mathbf{u} + b D_y \mathbf{u} 
		= P^{-1} ( a B_x + b B_y ) ( \mathbf{u} - \mathbf{g} ),
\end{equation} 
Here, $\mathbf{u} = [u_1,\dots,u_N]^T$ are the nodal values of the numerical solution at the nodes $S = \{\mathbf{x}_n\}_{n=1}^N$ and $\mathbf{g} = [g_1,\dots,g_N]^T$ is the boundary data vector with 
\begin{equation}\label{eq:g}
	g_n = 
	\begin{cases} 
		g(t,\mathbf{x}_n) & \text{if } \mathbf{x}_n \in \Gamma_-, \\ 
		u_n & \text{if } \mathbf{x}_n \in \Gamma_+.
	\end{cases}
\end{equation}
Note that (ii) in \cref{def:mimeticB} ensures that $B_x$ and $B_y$ only act on surface entries corresponding to points on the boundary $\Gamma$.\footnote{ 
We can therefore assign arbitrary values to the entries in $\mathbf{g}$ that correspond to points in the interior of $\Omega$. 
}
The right-hand side of \cref{eq:linear_discr} is a \emph{simultaneous approximation term (SAT)} that weakly enforces the BC by ``forcing" the numerical solution toward the boundary data $g$ at the inflow boundary part. 
See \cite{svard2014review,fernandez2014review,del2018simultaneous} and references therein for more details on SATs.

\begin{remark} 
	In general, each SAT term on the right-hand side of \cref{eq:linear_discr} involves a free penalty parameter as, for instance, $\sigma a P^{-1} B_x ( \mathbf{u} - \mathbf{g} )$ instead of $a P^{-1} B_x ( \mathbf{u} - \mathbf{g} )$. 
	We have fixed the penalty parameter to $\sigma = 1$, which was demonstrated to minimize boundary data errors in \cite[Appendix A]{glaubitz2022summation} (also see \cite[page 11]{aalund2016provably}). 
\end{remark}

\subsection{Conservation} 
\label{sub:linear_cons} 

The exact solution of \cref{eq:linear} satisfies the following \emph{conservation} property:
\begin{equation}\label{eq:conservation}
	\frac{\d}{\d t} \int_{\Omega } u \intd \boldsymbol{x} 
		= - \left( \int_{\Gamma_+}  u (a n_x + b n_y) \intd s + \int_{\Gamma_-}  g (a n_x + b n_y) \intd s \right).
\end{equation}  
This means that the total amount of the quantity $u$ is neither created nor destroyed inside the domain and only changes due to the flux across the boundaries. 
We now address how conservation can be mimicked in the discrete case using MFSBP operators. 
The left-hand side of \cref{eq:conservation} is approximated by $\mathbf{1}^T P \mathbf{u}_t$. 
Substituting the MFSBP-SAT semi-discretization \cref{eq:linear_discr}, we get 
\begin{equation}\label{eq:cons_aux1}
	\mathbf{1}^T P\mathbf{u}_t 
		= - a \mathbf{1}^T P D_x \mathbf{u} - b \mathbf{1}^T P D_y \mathbf{u} +  \mathbf{1}^T ( a B_x + b B_y ) ( \mathbf{u} - \mathbf{g} ).
\end{equation}
The MFSBP properties transform \cref{eq:cons_aux1} to 
\begin{equation}\label{eq:cons_aux2}
	\mathbf{1}^T P\mathbf{u}_t 
		= a (D_x \mathbf{1})^T P \mathbf{u} + b (D_y \mathbf{1})^T P \mathbf{u} - \mathbf{1}^T \left(a B_x + b B_y \right) \mathbf{g}.
\end{equation} 
The last term on the right-hand side of \cref{eq:cons_aux2} approximates the right-hand side of \cref{eq:conservation}. 
However, \cref{eq:cons_aux2} also contains the additional volume terms $a(D_x \mathbf{1})^T P \mathbf{u}$ and $b(D_y \mathbf{1})^T P \mathbf{u}$, while no such terms are present in \cref{eq:conservation}. 
To avoid artificial construction or destruction of the quantity $u$, 
\begin{equation}\label{eq:restr_Dx_Dy}
	D_x \mathbf{1} = D_y \mathbf{1} = \mathbf{0}
\end{equation}
must hold. 
This is satisfied by construction for polynomial MSBP operators since non-zero constants are polynomials of degree zero. 
For MFSBP operators, we can ensure \cref{eq:restr_Dx_Dy} by requiring the MFSBP operators to be exact for constants, i.e., $1 \in \mathcal{F}$. 
In this case, \cref{eq:cons_aux2} becomes 
\begin{equation}\label{eq:cons_aux3}
	\mathbf{1}^T P\mathbf{u}_t 
		= - \mathbf{1}^T \left(a B_x + b B_y \right) \mathbf{g}, 
\end{equation} 
which is the discrete analog to \cref{eq:conservation}.

\subsection{Energy stability} 
\label{sub:linear_energy} 

The exact solution of \cref{eq:linear} is \emph{energy-bounded} since it satisfies 
\begin{equation}\label{eq:stability}
\begin{aligned}
	\frac{\rm d}{\rm d t} \norm{u}_{L_2}^2 
		= - \int_{\Gamma_+} \left( a n_x + b n_y \right) u^2  \intd s - \int_{\Gamma_-} \left( a n_x + b n_y\right) g^2 \intd s,  
\end{aligned} 
\end{equation} 
where $\Gamma_+$ and $\Gamma_-$ are the outflow and inflow boundary part, respectively. 
A similar relation to \cref{eq:stability} is desired in the discrete setting to establish discrete energy stability for the numerical solution. 
We observe that the left-hand side \cref{eq:stability} is approximated by $\frac{\d}{\d t} \mathbf{u}^T P \mathbf{u}$. 
Substituting the MFSBP-SAT semi-discretization \cref{eq:linear_discr}, we get 
\begin{equation}\label{eq:stab_aux1} 
	\frac{1}{2} \frac{\d}{\d t} \mathbf{u}^T P \mathbf{u} 
		= \mathbf{u}^T P\mathbf{u}_t 
		= - a \mathbf{u}^T P D_x \mathbf{u} - b \mathbf{u}^T P D_y \mathbf{u} + \mathbf{u}^T ( a B_x + b B_y ) ( \mathbf{u} - \mathbf{g} ).
\end{equation}
The MFSBP properties transform \cref{eq:stab_aux1} to 
\begin{equation}\label{eq:stab_aux2}
	\frac{1}{2} \frac{\d}{\d t} \mathbf{u}^T P \mathbf{u} 
		= a \mathbf{u}^T D_x^T P \mathbf{u} + b \mathbf{u}^T D_y^T P \mathbf{u} - \mathbf{u}^T ( a B_x + b B_y ) \mathbf{g}.
\end{equation} 
Since $\mathbf{u}^T D_x^T P \mathbf{u}$ and $\mathbf{u}^T D_y^T P \mathbf{u}$ are real numbers, they are equal to their transposes, and summing \cref{eq:stab_aux1,eq:stab_aux2} yields 
\begin{equation}\label{eq:stab_aux3}
	\frac{\d}{\d t} \mathbf{u}^T P \mathbf{u} 
		= \mathbf{u}^T ( a B_x + b B_y ) ( \mathbf{u} - 2 \mathbf{g} ).
\end{equation} 
Assume that $B_x$ and $B_y$ are mimetic boundary operators (see \cref{def:mimeticB}). 
Using (i), (ii), and (iii) in \cref{def:mimeticB}, the right-hand side of \cref{eq:stab_aux3} reduces to 
\begin{equation}\label{eq:stab_aux4}
	\frac{\d}{\d t} \mathbf{u}^T P \mathbf{u} 
		= \sum_{m=1}^M u_m \left[ (v_x)_m a n_x(\mathbf{x}_m) + (v_y)_m b n_y(\mathbf{x}_m) \right] \left( u_m - 2 g_m \right),
\end{equation}
where $(b_x)_m = (v_x)_m n_x(\mathbf{x}_m)$ and $(b_y)_m = (v_y)_m n_y(\mathbf{x}_m)$ are the non-zero diagonal elements of $B_x$ and $B_y$, respectively. 
Note that the sum in \cref{eq:stab_aux4} only includes the surface nodes $\{ \mathbf{x}_m \}_{m=1}^M \subset \Gamma$. 
Suppose that $B_x$ and $B_y$ are associated with the same positive and $\mathcal{F}^2$-exact surface quadrature, i.e.\ $v_{m} = (v_x)_m = (v_y)_m$. 
Moreover, for ease of notation, we denote 
\begin{equation} 
	\tau_m = v_m \left[ a n_x(\mathbf{x}_m) + b n_y(\mathbf{x}_m) \right].
\end{equation} 
Note that $\tau_m \geq 0$ if $\mathbf{x}_m \in \Gamma_+$ and $\tau_m \leq 0$ if $\mathbf{x}_m \in \Gamma_-$. 
We can therefore decompose the sum in \cref{eq:stab_aux5} into two parts, 
\begin{equation}\label{eq:stab_aux5}
	\frac{\d}{\d t} \mathbf{u}^T P \mathbf{u} 
		= \sum_{\mathbf{x}_m \in \Gamma_+} \tau_m u_m \left( u_m - 2 g_m \right) 
		+ \sum_{\mathbf{x}_m \in \Gamma_-} \tau_m u_m \left( u_m - 2 g_m \right),
\end{equation}
where the first contains the surface points on the outflow boundary part ($\mathbf{x}_m \in \Gamma_+$) and the second contains the surface points on the inflow boundary part ($\mathbf{x}_m \in \Gamma_{-}$). 
By construction of $\mathbf{g}$ in \cref{eq:g}, we have $g_m = u_m$ in the first sum and $g_m = g(t,\mathbf{x_m})$ in the second sum. 
Hence, \cref{eq:stab_aux5} becomes  
\begin{equation}\label{eq:stab_aux6} 
	\frac{\d}{\d t} \mathbf{u}^T P \mathbf{u} 
		= - \sum_{\mathbf{x}_m \in \Gamma_+} \tau_m u_m^2 
		+ \sum_{\mathbf{x}_m \in \Gamma_-} \tau_m u_m \left( u_m - 2 g(t,\mathbf{x_m}) \right).
\end{equation} 
Note that $u_m ( u_m - 2 g(t,\mathbf{x_m}) ) = [ u_m - g(t,\mathbf{x_m}) ]^2 - g(t,\mathbf{x_m})^2$, and we can therefore rewrite \cref{eq:stab_aux6} as 
\begin{equation}\label{eq:stab_aux7} 
\begin{aligned}
	\frac{\d}{\d t} \mathbf{u}^T P \mathbf{u} 
		& = - \sum_{\mathbf{x}_m \in \Gamma_+} \tau_m u_m^2 
			- \sum_{\mathbf{x}_m \in \Gamma_-} \tau_m g(t,\mathbf{x_m})^2 
			+ \sum_{\mathbf{x}_m \in \Gamma_-} \tau_m \left[ u_m - g(t,\mathbf{x_m}) \right]^2. 
\end{aligned}
\end{equation} 
Recall that $\tau_m = v_m \left[ a n_x(\mathbf{x}_m) + b n_y(\mathbf{x}_m) \right]$. 
Hence, the first two sums on the right-hand side of \cref{eq:stab_aux7} are the discrete analog to \cref{eq:stability}. 
The additional third sum in \cref{eq:stab_aux7} corresponds to a non-negative damping term that stabilizes the semi-discretization and vanishes with increasing accuracy. 

%% file: 7_examples.tex
\section{Examples of MFSBP operators} 
\label{sec:examples}

We now exemplify the construction of MFSBP operators on a triangular and circular reference element $\Omega \subset \R^2$. 
In both cases, we consider the non-polynomial and polynomial function spaces 
\begin{equation}\label{eq:testFspace}
\begin{aligned} 
	\mathcal{F}_1 & = \lspan\left\{ 1, x, y, x^2, xy, y^2, x^3, x^2y, xy^2, y^3 \right\}, \\
	\mathcal{F}_2 & = \lspan \left\{ 1, x, y, \sin(\omega (x + y)), \cos(\omega (x + y)) \right\}, 
\end{aligned}
\end{equation} 
where $\omega > 0$ is a parameter that remains to be determined. 
Although the polynomial function space $\mathcal{F}_1$ has a higher dimension than the trigonometric function space $\mathcal{F}_2$, we found the trigonometric one to yield significantly more accurate numerical solutions for the specific tests carried out in \cref{sub:tests_linear} and \cref{sub:tests_steady}. 
Note that for both spaces, $\partial_x (\mathcal{F}_i^2) = \partial_y (\mathcal{F}_i^2)$, $i=1,2$. 
Hence, it suffices to find a positive and $\partial_x (\mathcal{F}_i^2)$-exact volume quadrature, for which we use equidistributed Halton points. 
Although using Halton points ensures the existence of the desired volume quadratures (see \cite{glaubitz2022constructing}), they are not necessarily optimal in the sense that the number of quadrature points can be notably larger than the dimension of the function space $\partial_x (\mathcal{F}_i^2)$. 
Constructing optimal (so-called Gaussian) quadratures for general multi-dimensional function spaces is an open problem. 
Our ambition is to further investigate the optimization and efficiency of our MFSBP operators. 
However, such efforts need to be tailored to specific function spaces and will be carried out in future work. 

\begin{remark}
	In the upcoming numerical tests, we deliberately choose examples that favor the non-polynomial operators in order to exemplify the flexibility of using other function spaces as in the MFSBP operators.
\end{remark}

\subsection{MFSBP operators on triangles} 
\label{sub:examples_triangle}

We start by exemplifying the construction of MFSBP operators on the triangular domain in \cref{fig:triangle}, given by 
\begin{equation}\label{eq:triangular_domain}
	\Omega = \left\{ \, (x,y) \in \R^2 \mid 0 \leq x \leq 1, \, 0 \leq y \leq 1-x \, \right\}. 
\end{equation}   
The boundary of the triangle $\Omega$ consists of three linear parts with outward pointing unit normal vectors $[0, -1]^T$, $(1/\sqrt{2}) [1,1]^T$, and $[-1,0]^T$. 
We need these normal vectors in combination with a positive and $\mathcal{F}_i^2$-exact surface quadrature to construct a mimetic boundary matrix $B_x$. 
The surface quadrature should be the same on all linear boundary parts, up to a length-dependent scaling factor, to ensure inter-element conservation in a multi-element/block setting. 
For the triangle, we found a positive and $\mathcal{F}_1^2$-exact surface quadrature using $8$ equidistant points on each linear part. 
(No nodes were placed on the corners of the triangle to avoid undefined normal vectors.) 
The same number of surface points was necessary to find a positive and $\mathcal{F}_2^2$-exact surface quadrature. 
To find a positive and $\partial_x( \mathcal{F}_1^2 )$-exact volume quadrature for the polynomial MSBP operator, we had to add another $21$ Halton-distributed points in the interior of $\Omega$, using a total number of $45$ points. 
To find a positive and $\partial_x( \mathcal{F}_2^2 )$-exact volume quadrature for the trigonometric MFSBP operator, on the other hand, we had to add $12$ Halton-distributed points in the interior of $\Omega$, using a total number of $32$ points. 
In both cases, we could not find desired volume quadratures using fewer Halton-distributed points in the interior of $\Omega$. 
We then found $\mathcal{F}_1$- and $\mathcal{F}_2$-exact MFSBP operators, respectively, following the construction procedure in \cref{sec:construction}. 
The point sets used to find these operators are illustrated in \cref{fig:triangle_trig,fig:triangle_poly}, respectively. 
Although the polynomial MSBP operator has a higher dimension and uses more points than the trigonometric MFSBP operator, we found the trigonometric $\mathcal{F}_2$-exact MFSBP operator to yield significantly more accurate results for our numerical tests in \cref{sub:tests_linear}. 
To allow for a fair comparison between the polynomial and non-polynomial function space, we did not optimize the point distribution for any of the function spaces. 
As stated earlier, optimizing the point distribution, e.g., to get minimal quadratures, needs to be tailored to specific function spaces and will be addressed in future work. 

\begin{figure}[tb]
	\centering 
	\begin{subfigure}{0.49\textwidth}
		\includegraphics[width=\textwidth]{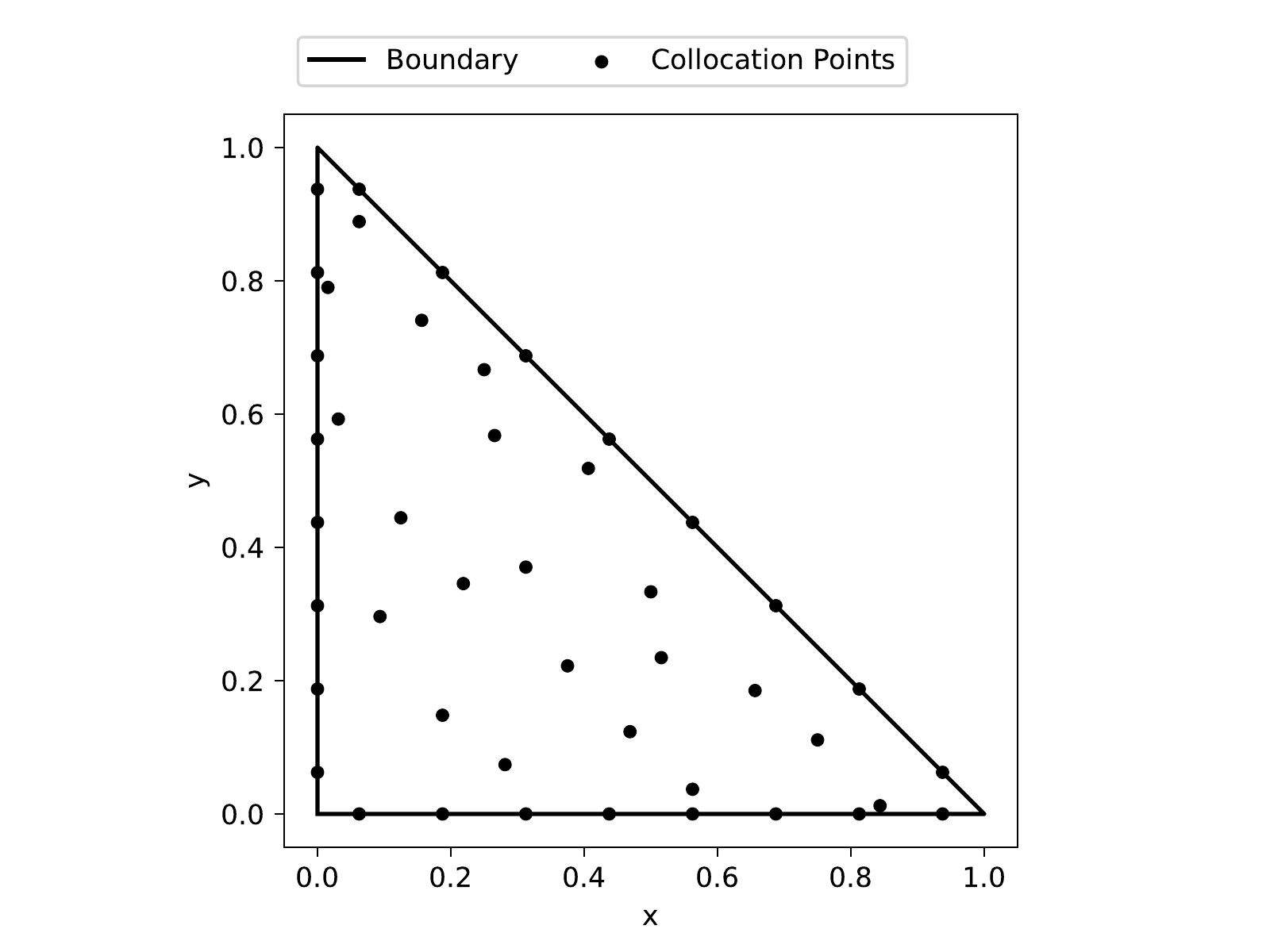}
		\caption{The 24 surface and 21 inner points used to construct the polynomial $\mathcal{F}_1$-exact MSBP operator} 
		\label{fig:triangle_poly}
	\end{subfigure}
	~ 	
	\begin{subfigure}{0.49\textwidth}
		\includegraphics[width=\textwidth]{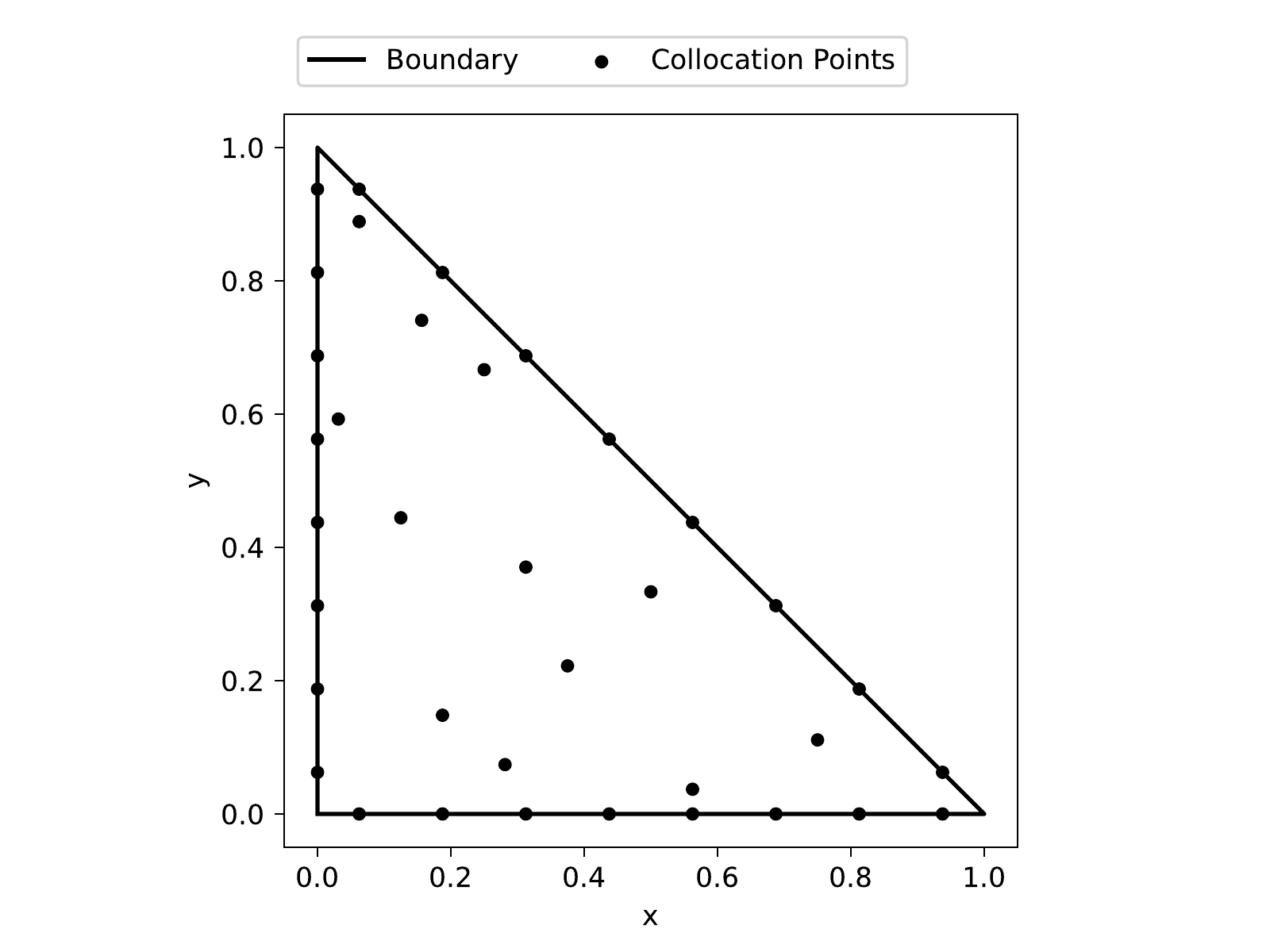}
		\caption{The 24 surface and 12 inner points used to construct the trigonometric $\mathcal{F}_2$-exact MFSBP operator} 
		\label{fig:triangle_trig}
	\end{subfigure}
	\caption{
	The surface and inner points used to construct the polynomial $\mathcal{F}_1$-exact and trigonometric $\mathcal{F}_2$-exact MFSBP operators. 
	The surface points are equidistant, and the inner points are Halton-distributed.
	}
	\label{fig:triangle}
\end{figure}

\subsection{MFSBP operators on the circle} 
\label{sub:examples_circle}

Next, we exemplify the construction of polynomial $\mathcal{F}_1$-exact and trigonometric $\mathcal{F}_2$-exact MFSBP operators, with the function spaces $\mathcal{F}_1$ and $\mathcal{F}_2$ as in \cref{eq:testFspace}, for the circular domain 
\begin{equation}
	\Omega = \left\{ \, \mathbf{x} \in \R^2 \mid \norm{\mathbf{x} - \mathbf{c}} \leq r \, \right \}
\end{equation}
with center $\mathbf{c} = [1/2, 1/2]^T$ and radius $r = 1/2$, see \cref{fig:circle}. 
The boundary of the disc consists of a single smooth part with a smoothly varying outward pointing unit normal vector $\boldsymbol{n}$. 
For the polynomial function space, we found a positive and $\mathcal{F}_1^2$-exact surface quadratures using $12$ equidistant points on the surface.
To next find a positive and $\partial_x( \mathcal{F}_1^2 )$-exact volume quadrature, following the construction procedure in \cref{sec:construction}, we had to add another $25$ points in the interior of $\Omega$, using a total number of $37$ points. 
\cref{fig:circle} illustrates the position of these points. 
We could not find desired volume quadratures using fewer Halton-distributed points in the interior of $\Omega$. 
For the trigonometric function space, we found a positive and $\mathcal{F}_2^2$-exact surface quadratures and a positive and $\partial_x( \mathcal{F}_1^2 )$-exact volume quadrature using the same points as for the polynomial $\mathcal{F}_1$-exact MSBP operator. 
Although the polynomial and trigonometric MFSBP operators use the same points, we found the trigonometric $\mathcal{F}_2$-exact MFSBP operator to yield significantly more accurate results for our numerical tests in \cref{sub:tests_disk}. 

\begin{figure}[tb]
	\centering
	\includegraphics[width=0.6\textwidth]{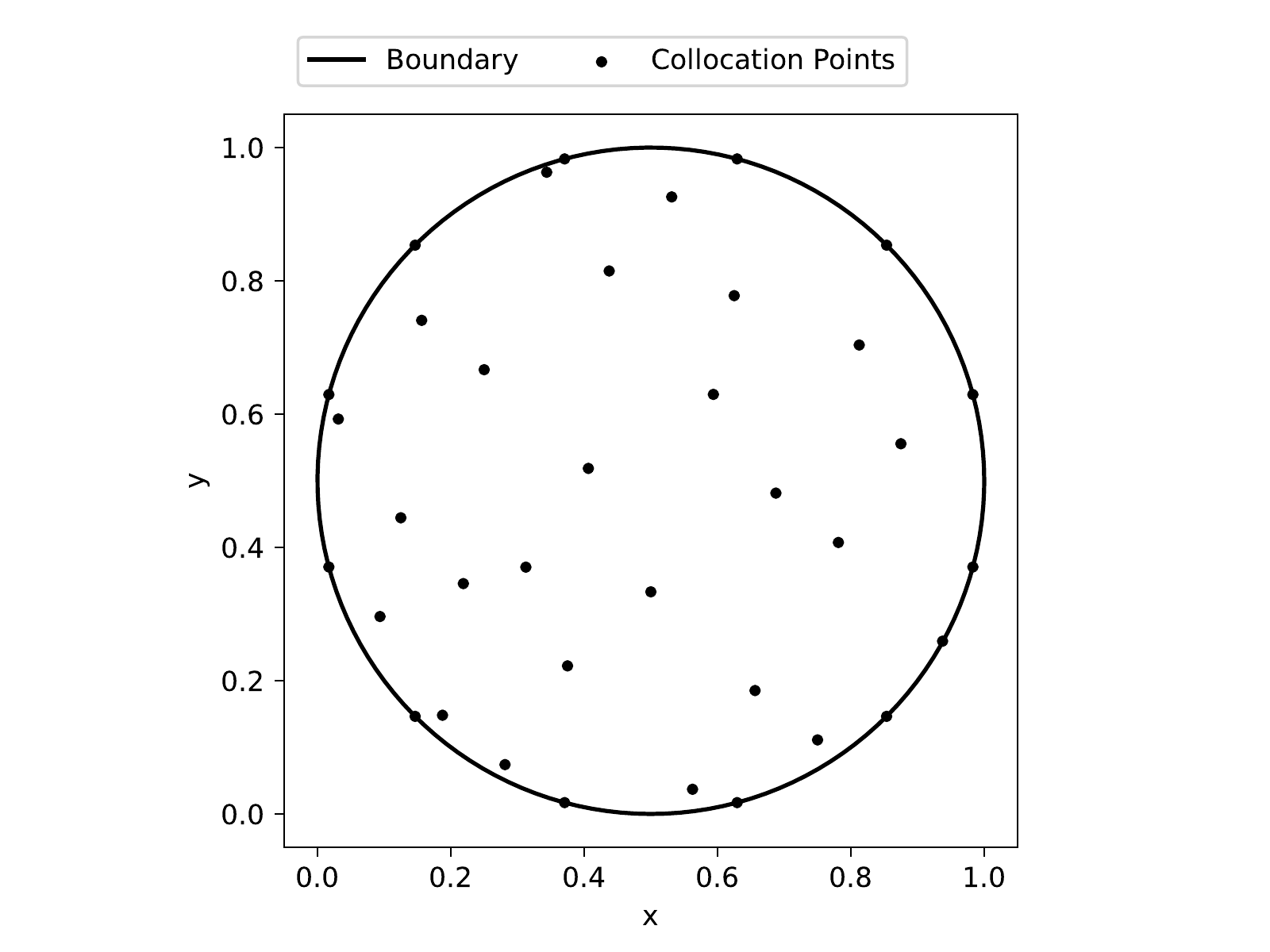}
	\caption{
	The 12 surface and 25 inner points used to construct the polynomial $\mathcal{F}_1$-exact and trigonometric $\mathcal{F}_2$-exact operators. 
	The surface points are equidistant, and the inner points are Halton-distributed.
	}
	\label{fig:circle}
\end{figure}

\begin{remark} 
	Although non-polygonal elements are rarely used in the context of SE-like methods for hyperbolic conservation laws, they are crucial in mesh-free RBF methods \cite{fasshauer2007meshfree,fornberg2015primer}. 
	In the recent work \cite{glaubitz2022energy}, we established a stability theory for global RBF methods using FSBP operators. 
	The MFSBP operators introduced here will allow us to extend this theory to the genuine multi-dimensional case, which will be addressed in future work, together with the extension to local RBF methods. 
\end{remark} 

%% file: 8_num_test.tex
\section{Numerical results} 
\label{sec:num_tests} 

We focus on the homogeneous and inhomogeneous linear advection equation because it provides prototypical examples for which non-polynomial approximation spaces are advantageous. 
Future work will consider more sophisticated problems.  
We used the explicit strong stability preserving (SSP) Runge--Kutta (RK) method of third order using three stages (SSPRK(3,3)) \cite{shu1988total} for all subsequent numerical tests. 
Furthermore, all tests are performed for conforming triangular grids. 
All tests were run with the elements scaled to be of unit length.
The Julia code used to generate the numerical tests presented here is open access and can be found on GitHub.\footnote{See \url{https://github.com/simonius/MFSBP}}

\subsection{The linear advection equation}
\label{sub:tests_linear}	

Consider the IBVP for the linear advection equation with constant coefficients, 
\begin{equation}\label{eq:testslinear}
\begin{aligned}
	\partial_t u+ a \partial_x u+ b \partial_y u & = 0, \quad && \forall (x,y) \in \Omega, \ t>0, \\ 
	u(0,x,y) & = u_0(x,y), \quad && \forall (x,y) \in \Omega, \\
	u(t,x,y) & = g(t,x,y), \quad && \forall (x,y) \in \Gamma_{-}, \ t \geq 0, 
\end{aligned}
\end{equation}
where $\Omega = [0, 1]^2$ and $\Gamma_{-} = \{ \, (x,y) \in \partial\Omega \mid a n_x+ b n_y<0 \, \}$ is the inflow part of the boundary of $\Omega$. 
In our numerical tests, we triangulated $\Omega$ by generating $K^2$ equally-sized squares and then decomposing each square into two equally-sized triangles. 
Given MFSBP operators $D_x=P^{-1} Q_x$ and $D_y=P^{-1} Q_y$ with mimetic boundary operators $B_x$ and $B_y$, we used the multi-block semi-discretization 
\begin{equation} \label{eq:LinAdvMBDisc}
	\derd{\mathbf{u}}{t} + a D_x \mathbf{u} + b D_y \mathbf{u} = P^{-1} B_x ( a\mathbf{u} - \fnum_x) + P^{-1} B_y (b\mathbf{u} - \fnum_y),
\end{equation}
where $\fnum_x$ and $\fnum_y$ are the numerical flux functions in the $x$- and $y$-direction, respectively, coupling neighboring blocks and weakly enforcing the boundary condition $u(t,x,y) = g(t,x,y)$ at the inflow part of the computation domain $\Omega$. 
Recall that the boundary operators $B_x$ and $B_y$ include the components of the outward pointing unit normal $\boldsymbol{n} = [n_x,n_y]^T$. 
We used the classical local Lax--Friedrichs (LLF) flux \cite{Lax1954Weak}, i.e, 
\begin{equation}\label{eq:llf} 
\begin{aligned}
	\fnum_x 
		& = \frac{a}{2} \left( u_i + u_o \right) - \mathrm{sign}(n_x) \frac{ c_{\mathrm{max}} }{2} \left( u_o - u_i \right), \\ 
	\fnum_y 
		& = \frac{b}{2} \left( u_i + u_o \right) - \mathrm{sign}(n_y) \frac{ c_{\mathrm{max}} }{2} \left( u_o - u_i \right),
\end{aligned}
\end{equation} 
where $u_i$ and $u_o$ are the inner and outer states at the two sides of the interface of the present block. 
Furthermore, $c_{\mathrm{max}}$ is an upper bound of the advection speed, ensuring that the LLF fluxes in \cref{eq:llf} are monotonic. 
For the linear advection equation, we chose $c_{\mathrm{max}} = \max\{ |a|, |b| \}$. 

\begin{figure}[tb]
	\centering 
	\begin{subfigure}{0.475\textwidth}
		\includegraphics[width=\textwidth]{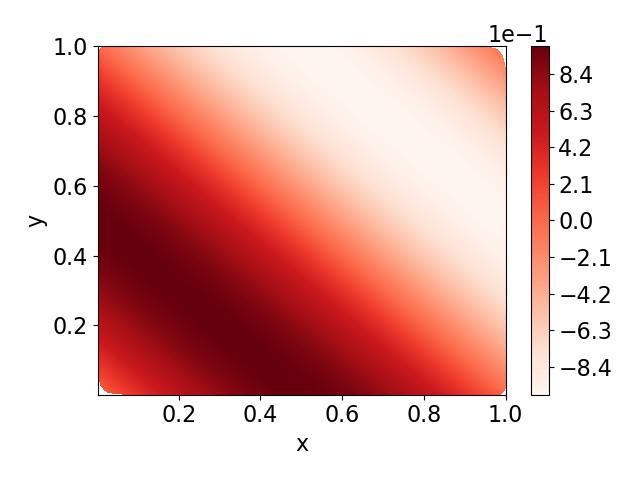}
		\caption{
		Numerical solution for the polynomial operator
		}
		\label{fig:TriLinAdvTest_solG}
	\end{subfigure}
	\begin{subfigure}{0.475\textwidth}
		\includegraphics[width=\textwidth]{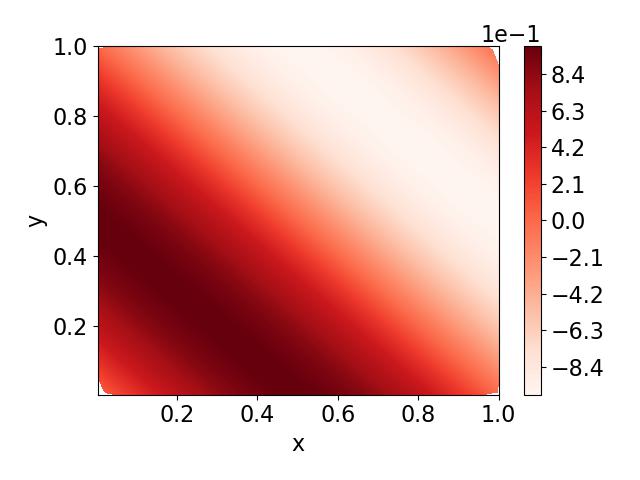}
		\caption{
		Numerical solution for the trigonometric operator
		}
		\label{fig:TriLinAdvTest_solF}
	\end{subfigure}
	\\
	\begin{subfigure}{0.475\textwidth}
		\includegraphics[width=\textwidth]{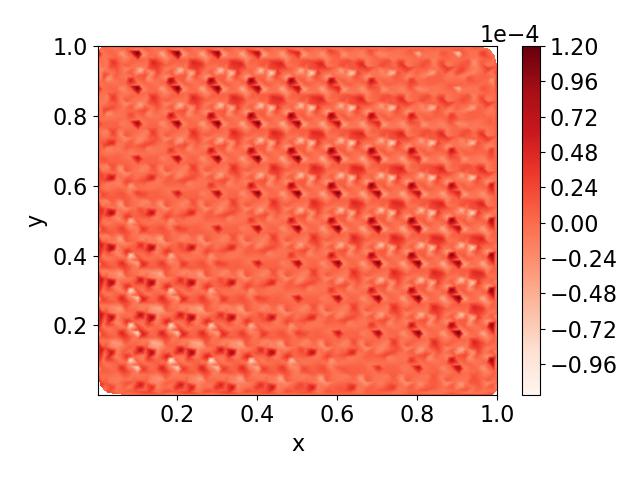}
		\caption{
		Point-wise error for the polynomial operator
		}
		\label{fig:TriLinAdvTest_errorG}
	\end{subfigure}
	\begin{subfigure}{0.475\textwidth}
		\includegraphics[width=\textwidth]{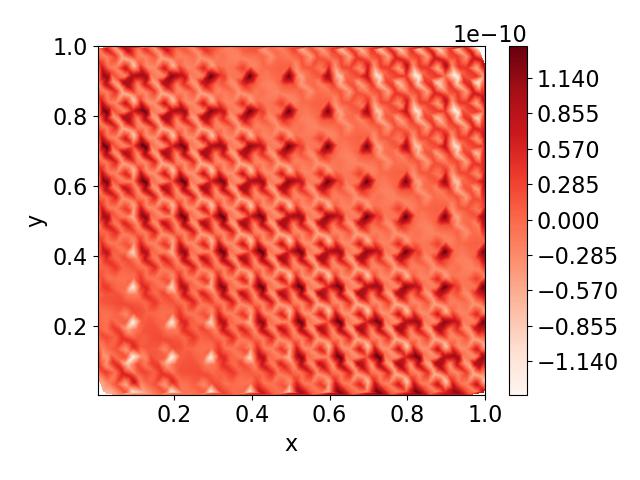}
		\caption{
		Point-wise error for the trigonometric operator
		}
		\label{fig:TriLinAdvTest_errorF}
	\end{subfigure}
	\caption{ 
	Numerical solutions and their point-wise errors of the IBVP \cref{eq:testslinear} at $t=1$ using the polynomial $\mathcal{F}_1$-exact and trigonometric $\mathcal{F}_2$-exact MFSBP operators. 
	We triangulated $\Omega$ by generating $K^2$ equally-sized squares with $K=10$ and then decomposing each square into two equally-sized triangles.
	Although $\mathcal{F}_1$ is ten-dimensional and $\mathcal{F}_2$ is five-dimensional, the trigonometric $\mathcal{F}_2$-exact MFSBP operator yields significantly more accurate results. 
	The contour plots do not fully cover the corners of the computational domain due to the used plotting routine. 
	In particular, to circumvent the issue of multivalued normals (see \cref{rem:corners}), we refrained from positioning nodes at the corners of the triangular reference element, thereby forgoing any extrapolation to these points.
	}
	\label{fig:TriLinAdvTest}
\end{figure}

We consider \cref{eq:testslinear} with constant velocity coefficients $a=b=1$, smooth initial condition $u_0(x, y) = \sin\left( \pi (x+y)\right)$, and inflow boundary data $g(x,y,t) = \sin\left( \pi ( x + y - (a+b) t ) \right)$ at the inflow part of the boundary $\Gamma_-$. 
The exact reference solution is 
\begin{equation}\label{eq:AdvRefSol}
	u(x, y, t) 
		= u_0\left(x-at, y-bt\right) 
		= \sin\left( \pi ( x + y )  - 2 \pi t \right)
\end{equation} 
for $a=b=1$. 
\cref{fig:TriLinAdvTest} illustrates the numerical solutions and their point-wise errors at time $t=1$ for the polynomial $\mathcal{F}_1$-exact and trigonometric $\mathcal{F}_2$-exact MFSBP operators. 
Both function spaces are described in \cref{eq:testFspace}, where we chose $\omega = \pi$ as the frequency parameter in $\mathcal{F}_2$. 
Recall from \cref{eq:testFspace} that the polynomial function space $\mathcal{F}_1$ is ten-dimensional, while the trigonometric one $\mathcal{F}_2$ is five-dimensional. 
Still, comparing the errors in \cref{fig:TriLinAdvTest_solG,fig:TriLinAdvTest_errorG}, the trigonometric $\mathcal{F}_2$-exact MFSBP operator yields significantly more accurate results.
We triangulated $\Omega$ by generating $K^2$ equally-sized squares with $K=10$ and then decomposing each square into two equally-sized triangles.
	
\begin{figure}[tb]
	\centering
	\includegraphics[width=0.6\textwidth]{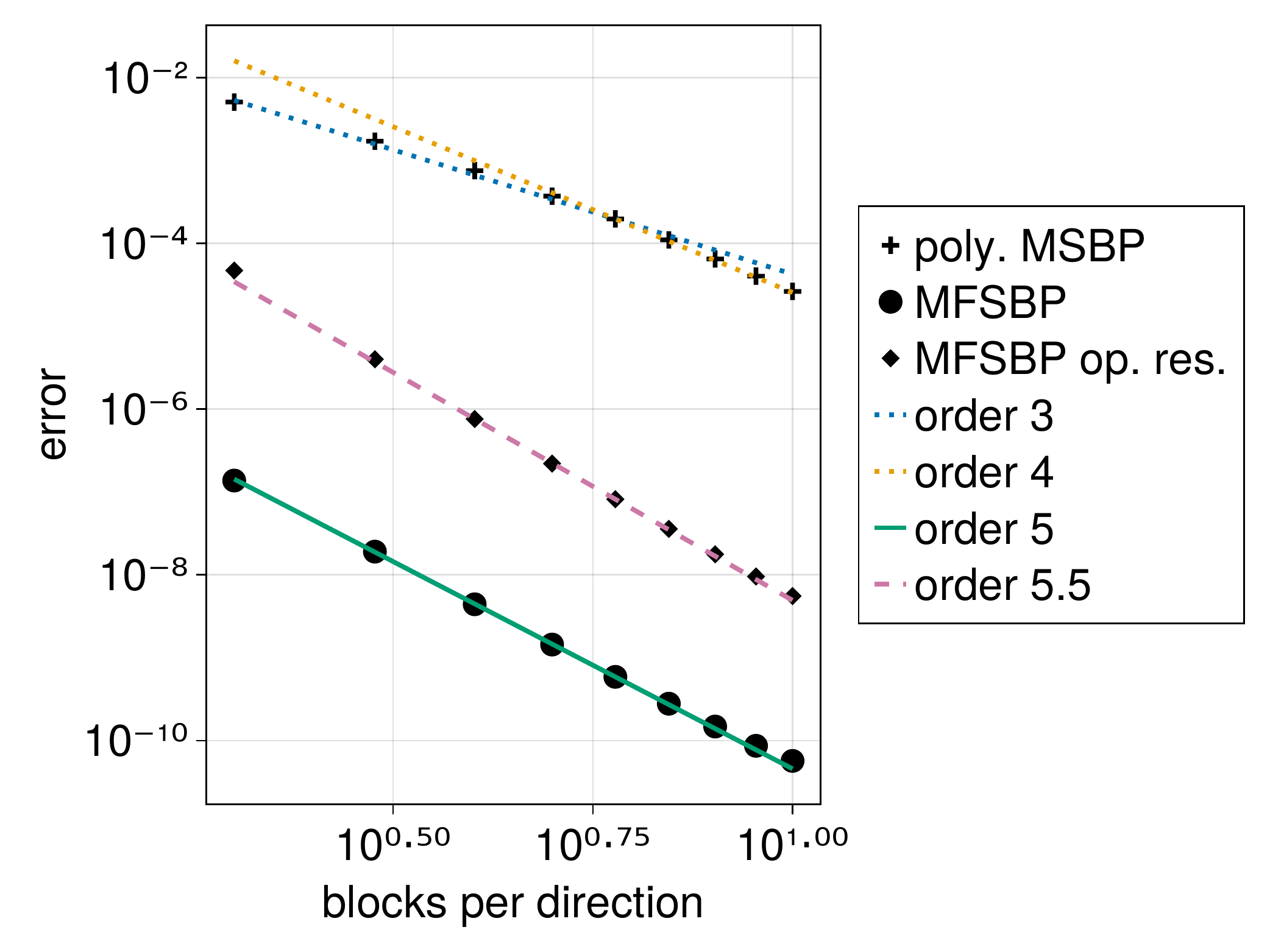}
	\caption{ 
	Errors of the numerical solutions to the IBVP \cref{eq:testslinear} at $t=1$ using the polynomial $\mathcal{F}_1$-exact (plusses along the dotted lines) and trigonometric $\mathcal{F}_2$-exact (dots along the straight line) MFSBP operators, along the residual \cref{eq:residual} of the trigonometric MFSBP operator (diamonds along the dashed line), for an increasing number of triangular blocks. 
	The trigonometric MFSBP operator yields significantly smaller errors than the polynomial one. 
	}
	\label{fig:convana}
\end{figure}

We next provide a brief error analysis, further comparing the polynomial $\mathcal{F}_1$-exact and trigonometric $\mathcal{F}_2$-exact MFSBP operators. 
To this end, we performed the same test for the IBVP \cref{eq:testslinear} with different numbers of triangular blocks. 
\cref{fig:convana} reports on the corresponding errors. 
Thereby, for $K=3,\dots,10$, we again considered the domain $\Omega = [0, 1]^2$, triangulized by first partitioning $\Omega$ into $K^2$ equally-sized squares and then decomposing each square into two equally-sized triangles. 
The errors in \cref{fig:convana} were calculated as 
\begin{equation}\label{eq:tests_error}
	E = \sqrt{ \frac{1}{N \abs{\mathcal{T}}} \sum_{T \in \mathcal{T}} \sum_{n = 1}^N \abs{u\left( x_{n}^{(T)}, y_{n}^{(T)}, t \right) - u_{n}^{(T)} (t) }^2 },
\end{equation}  
where $\mathcal{T}$ is the collection of all triangular elements and $u\left( x_{n}^{(T)}, y_{n}^{(T)}, t \right)$ and $u_{n}^{(T)} (t)$ denote the nodal value of the reference and numerical solution at the $N$ grid points in the elements $T \in \mathcal{T}$, respectively. 
The errors reported in \cref{fig:convana} were obtained for a fixed time step size $\Delta t = 10^{-3}$. 
We used a fixed time step size to focus solely on the errors of the spatial semi-discretizations while making errors introduced by the SSPRK(3,3) time integration method neglectable. 
\cref{fig:convana} demonstrates that the polynomial $\mathcal{F}_1$-exact MSBP operator yields third- to fourth-order convergence. 
At the same time, we observe fifth-order convergence for the trigonometric $\mathcal{F}_2$-exact MFSBP operator. 
A possible explanation for the observed, higher than expected, convergence rate of the trigonometric MFSBP operator is the vanishing residual of the operator, 
\begin{equation}\label{eq:residual}
	\norm{D_x V - V_x}_2, 
\end{equation}
which is shown in \cref{fig:convana} to converge to zero with a similar order. 
Here, $\| A \|_2 = \sqrt{ \sum_{n,m=1}^N [A]_{n,m}^2 }$ is the Frobenius norm of the matrix $A \in \R^{N \times N}$. 
A more detailed investigation of the convergence rate of the proposed MFSBP operators will be provided in future works. 

\begin{remark} 
	We have conducted the above experiments with various advection directions, diverging from the previously used $a=b=1$. 
	The different advection directions do not noticeably influence the performance of the MFSBP-based scheme. 
	This is because the solution is a trigonometric function and is therefore well represented by the function space $\mathcal{F}_2$, regardless of the advection direction.
\end{remark}

\subsection{Including a radial basis function}

We now revisit the test problem described in \cref{sub:tests_linear}, this time employing an MFBP operator on the triangle that is exact for a function space including a Gaussian RBF. 
This demonstration serves two objectives: 
Firstly, it provides an example of a non-polynomial function space, beyond trigonometric functions. 
Secondly, it underscores the notion that the efficacy of MFSBP-based methodologies relies on the degree to which the function space $\mathcal{F}$ encapsulates the true underlying solution.
To elaborate, note that the solution for the previously considered test problem in \cref{sub:tests_linear} was a trigonometric function. 
Consequently, the trigonometric function space $\mathcal{F}_2$ mirrors this characteristic.
In contrast, we now examine a four-dimensional function space including an RBF function, given by
\begin{equation}
	\mathcal{F}_3 = \lspan \left\{1, x, y, \exp\left(- \| x-x_0 \|^2 / d^2 \right) \right\}.
\end{equation}
This function space operates on the reference triangle, with the RBF centered at $x_0 = (1/3, 1/3)$ and displaying a characteristic diameter of $d = 1/5$.
Despite our ability to construct an $\mathcal{F}_3$-exact MFSBP operator (see \cref{rem:F3} for more details on the construction), the resultant numerical solution is less accurate than that derived using the trigonometric $\mathcal{F}_2$-exact MFSBP or the polynomial $\mathcal{F}_1$-exact MSBP operator from before.
This disparity in accuracy is depicted in \cref{fig:LinRBFsol}, which presents the numerical solution and point-wise error for the $\mathcal{F}_3$-exact MFSBP operator, juxtaposed against the test problem from \cref{fig:TriLinAdvTest_errorF}.
As before, we triangulated $\Omega$ by generating $K^2$ equally-sized squares with $K=10$ and then decomposing each square into two equally-sized triangles.
As anticipated, we see that the $\mathcal{F}_3$-exact MFSBP operator yields less accurate results than the trigonometric $\mathcal{F}_2$-exact MFSBP as well as the polynomial $\mathcal{F}_1$-exact MSBP operator. 
This can be attributed to $\mathcal{F}_3$'s less effective representation of the true underlying solution. 

\begin{figure}
	\begin{subfigure}{0.48\textwidth}
		\includegraphics[width=\textwidth]{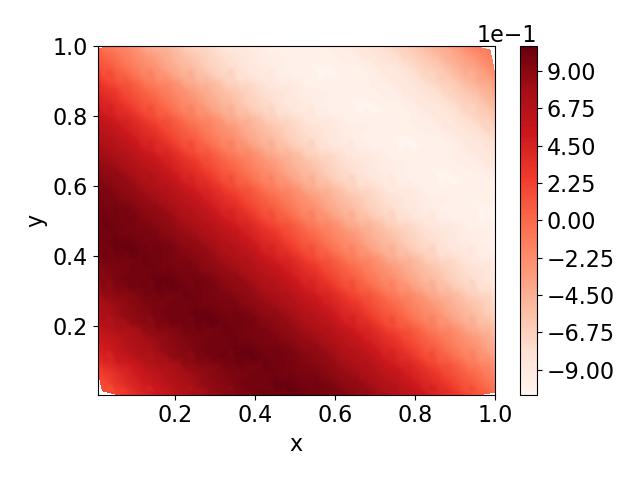}
	\caption{Numerical solution for the RBF operator}
	\end{subfigure}
	\begin{subfigure}{0.48\textwidth}
		\includegraphics[width=\textwidth]{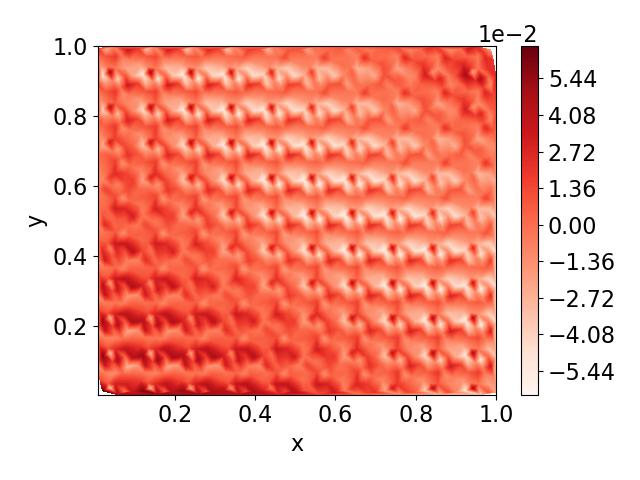}
	\caption{Point-wise error for the RBF operator}
	\end{subfigure}
	\caption{
	The numerical solution and its point-wise errors of the IBVP \cref{eq:testslinear} at $t=1$ using the $\mathcal{F}_3$-exact MFSBP operators. 
	We triangulated $\Omega$ by generating $K^2$ equally-sized squares with $K=10$ and then decomposing each square into two equally-sized triangles.
	The contour plots do not cover the corners of the computational domain due to the used plotting routine.
	}
	\label{fig:LinRBFsol}
\end{figure}

\begin{remark}\label{rem:F3}
	We found a positive and $\mathcal{F}_3^2$-exact surface quadrature using seven equidistant points on each linear part. 
	(Again, no nodes were placed on the corners of the triangle to avoid undefined normal vectors.) 
	Observe that $\partial_x (\mathcal{F}_3^2) = \partial_y (\mathcal{F}_3^2)$ and it therefore suffices to additionally find a positive $\partial_x (\mathcal{F}_3^2)$-exact volume quadrature to get an MFSBP operator. 
	To find such a volume quadrature, we used 9 Halton-distributed points in the triangle's interior. 
\end{remark}

\subsection{The inhomogeneous linear advection equation}
\label{sub:tests_steady}

We next consider the IBVP for the \emph{inhomogenous} linear advection equation with constant coefficients, 
\begin{equation}\label{eq:testslinear2}
\begin{aligned}
	\partial_t u+ a \partial_x u+ b \partial_y u & = s(u,x,y), \quad && \forall (x,y) \in \Omega, \ t>0, \\ 
	u(0,x,y) & = u_0(x,y), \quad && \forall (x,y) \in \Omega, \\
	u(t,x,y) & = g(t,x,y), \quad && \forall (x,y) \in \Gamma_{-}, \ t \geq 0, 
\end{aligned}
\end{equation}
with source term $s(u,x,y) = \omega \cos(\omega x)\sin(\omega y) + \omega \sin(\omega x) \cos(\omega y) + \sin(\omega x) + \sin(\omega y) - u$.
We consider \cref{eq:testslinear2} on $\Omega = [0, 1]^2$ and with constant velocity coefficients $a=b=1$. 
The generated grid was again constructed by first partitioning $\Omega$ into $K^2$ equally-sized squares before decomposing each square into two equally-sized triangles. 
We used zero initial and inflow boundary conditions, $u_0 \equiv 0$ and $g \equiv 0$. 
For these choices, \cref{eq:testslinear2} admits the trigonometric steady state solution $u(x, y) = \sin(\omega x) \sin(\omega y)$. 
We solved \cref{eq:testslinear2} again using the MFSBP semi-discretization \cref{eq:LinAdvMBDisc}. 
This time, we compare the polynomial $\mathcal{F}_1$-exact MSBP operator (see \cref{eq:testFspace} for the definition of $\mathcal{F}_1$) with a trigonometric $\mathcal{F}_4$-exact MFSBP operator in the triangular domain \cref{eq:triangular_domain}. 
Here, $\mathcal{F}_4$ is a seven-dimensional mixed polynomial-trigonometric function space defined as
\begin{equation}
	\mathcal{F}_4 
		= \lspan\left\{ 1, x, y, \sin(\omega x)\sin(\omega y), \cos(\omega x) \cos(\omega y), \sin(\omega x) \cos(\omega y), \cos(\omega x) \sin(\omega y) \right \},
\end{equation}
where we set $\omega = 2 \pi $ in our numerical tests.  
For the triangle, we found a positive and $\mathcal{F}_4^2$-exact surface quadrature using 12 equidistant points on each linear part. 
(Again, no nodes were placed on the corners of the triangle to avoid undefined normal vectors.) 
Observe that $\partial_x (\mathcal{F}_4^2) = \partial_y (\mathcal{F}_4^2)$ and it therefore suffices to additionally find a positive $\partial_x (\mathcal{F}_4^2)$-exact volume quadrature to get an MFSBP operator. 
To find such a volume quadrature, we used 24 Halton-distributed points in the triangle's interior. 

\begin{figure}[tb]
	\centering
	\begin{subfigure}{0.475\textwidth}
		\includegraphics[width=\textwidth]{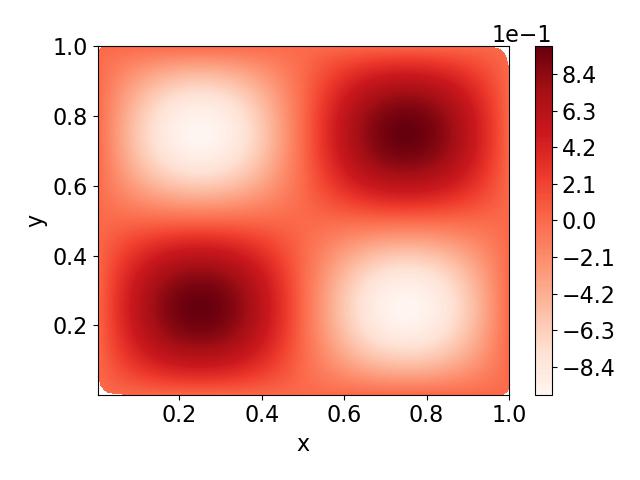}
		\caption{
		Numerical solution for the polynomial operator
		}
		\label{fig:steadySol_solG}
	\end{subfigure}
	\begin{subfigure}{0.475\textwidth}
		\includegraphics[width=\textwidth]{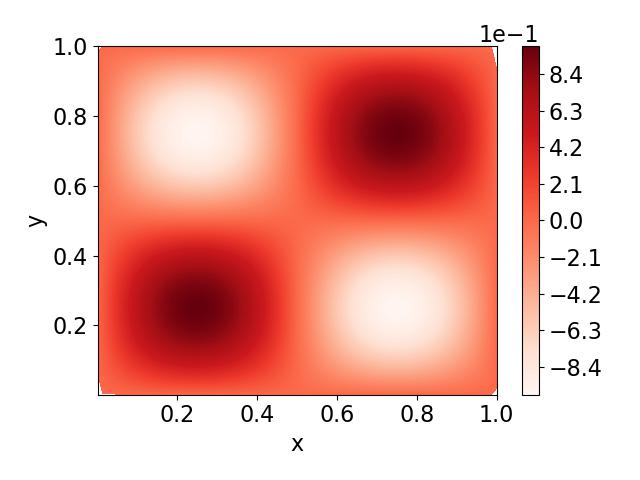}
		\caption{
		Numerical solution for the trigonometric operator
		}
		\label{fig:steadySol_solH}
	\end{subfigure}
	\\
	\begin{subfigure}{0.475\textwidth}
		\includegraphics[width=\textwidth]{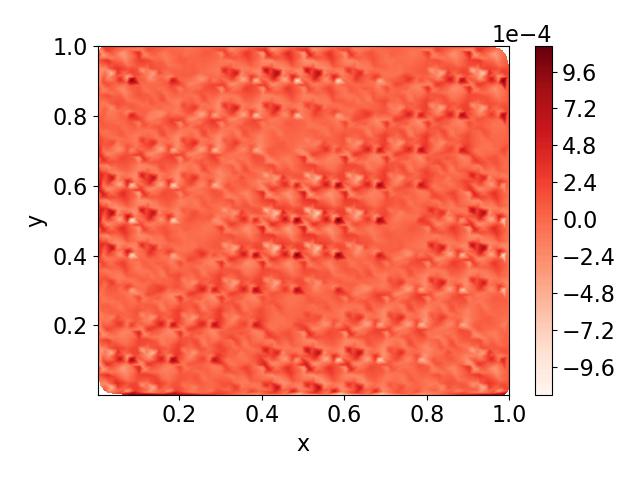}
		\caption{
		Point-wise errors for the polynomial operator
		}
		\label{fig:steadySol_errorG}
	\end{subfigure}
	\begin{subfigure}{0.475\textwidth}
		\includegraphics[width=\textwidth]{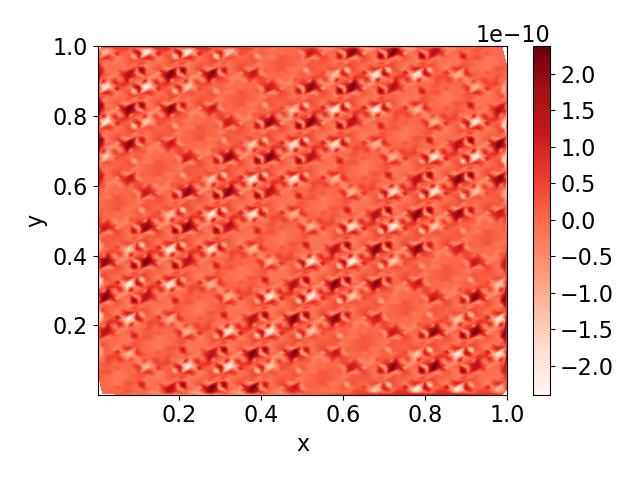}
		\caption{
		Point-wise errors for the trigonometric operator
		}
		\label{fig:steadySol_errorH}
	\end{subfigure}
	\caption{  
	Numerical solutions and their point-wise errors of the IBVP \cref{eq:testslinear2} at $t=2$ using the polynomial $\mathcal{F}_1$-exact and trigonometric $\mathcal{F}_4$-exact MFSBP operators. 
	Although $\mathcal{F}_1$ is ten-dimensional and $\mathcal{F}_4$ is seven-dimensional, the trigonometric $\mathcal{F}_4$-exact MFSBP operator yields significantly more accurate results. 
	We triangulated $\Omega$ by generating $K^2$ equally-sized squares with $K=10$ and then decomposing each square into two equally-sized triangles. 
	The contour plots do not cover the corners of the computational domain due to the used plotting routine. 
	}	
	\label{fig:steadySol}
\end{figure}

\cref{fig:steadySol} illustrates the numerical solutions and their point-wise errors at time $t=2$ for the polynomial $\mathcal{F}_1$-exact and trigonometric $\mathcal{F}_4$-exact MFSBP operators. 
Although the polynomial space $\mathcal{F}_1$ is ten-dimensional and the trigonometric space $\mathcal{F}_4$ is seven-dimensional, we see in \cref{fig:steadySol} that the trigonometric MSFBP operator yields a significantly more accurate numerical solution. 
We triangulated $\Omega$ by generating $K^2$ equally-sized squares with $K=10$ and then decomposing each square into two equally-sized triangles.

\begin{figure}[tb] 
	\centering
	\includegraphics[width=0.6\textwidth]{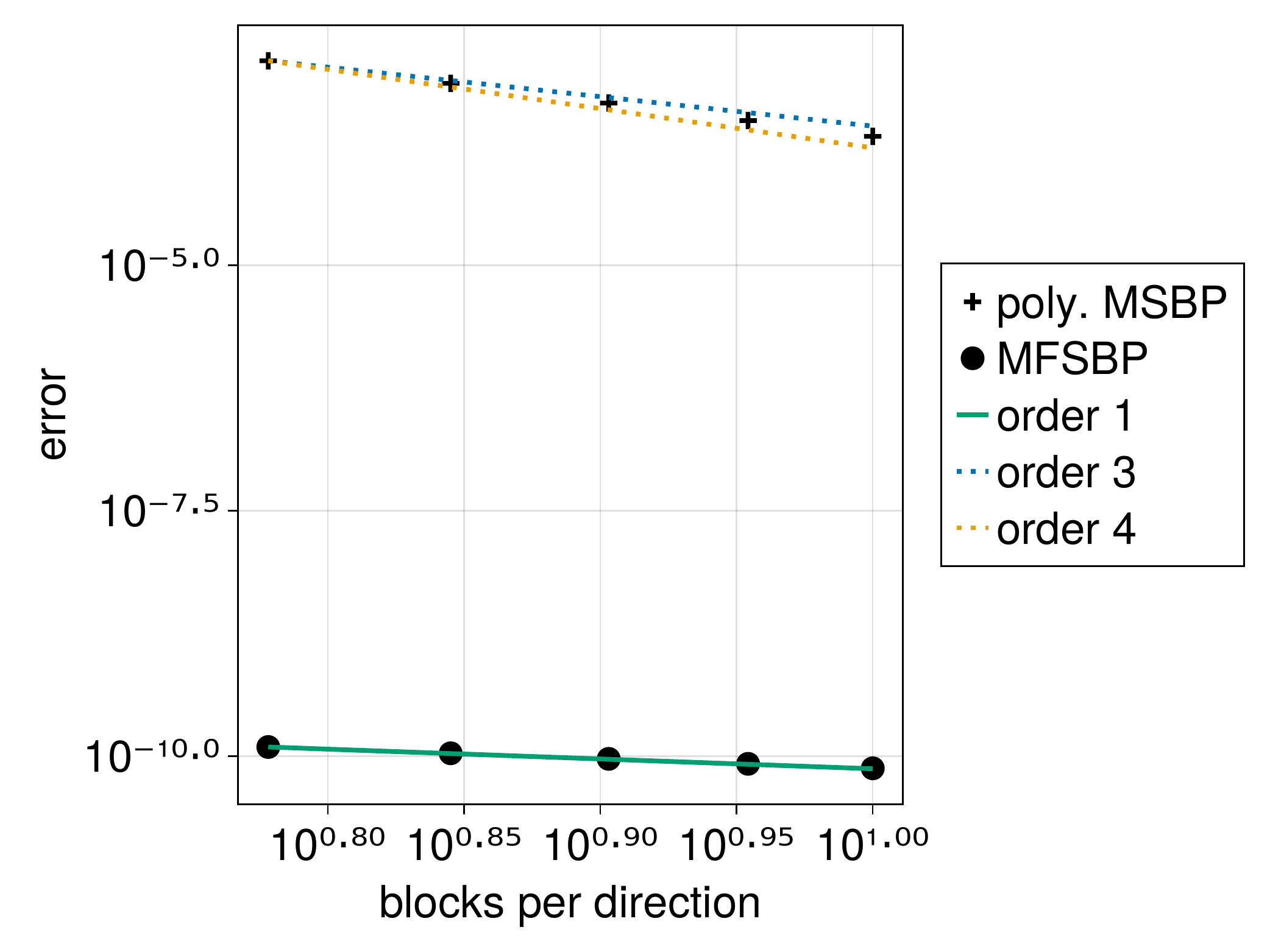}
	\caption{
	Errors of the numerical solutions to the IBVP \cref{eq:testslinear2} at $t=2$ using the polynomial $\mathcal{F}_1$-exact (plusses along the dotted lines) and trigonometric $\mathcal{F}_4$-exact (dots along the straight line) MFSBP operators for an increasing number of triangular blocks. 
	The trigonometric MFSBP operator yields significantly smaller errors than the polynomial one. 
	}
	\label{fig:sconvana}
\end{figure}

We again follow up with a brief error analysis. 
To this end, we performed the same test for the IBVP \cref{eq:testslinear2} with different numbers of triangular blocks. 
\cref{fig:sconvana} reports on the corresponding errors of the polynomial $\mathcal{F}_1$-exact and trigonometric $\mathcal{F}_4$-exact MFSBP operators. 
As before, the polynomial $\mathcal{F}_1$-exact MSBP operator yields third- to fourth-order convergence. 
At the same time, in all cases, the trigonometric $\mathcal{F}_4$-exact MFSBP operator yields errors around $10^{-10}$, which change only slightly as the number of blocks increases. 
We suspect that the reason for this is that the trigonometric space $\mathcal{F}_4$ can exactly represent the exact steady-state solution. 
Hence, we obtain the exact solution module round-off errors using the $\mathcal{F}_4$-exact MFSBP operator. 
Increasing the number of blocks does not significantly decrease these round-off errors. 
Similar results were reported in \cite{yuan2006discontinuous} for DG methods based on non-polynomial approximation spaces---although without the SBP property and the resulting discrete stability of the scheme.

\subsection{Advection on a circular disk}
\label{sub:tests_disk}

We end this section by demonstrating the geometric flexibility of MFSBP operators. 
To this end, 
we consider the IBVP \cref{eq:testslinear} for the linear advection equation from \cref{sub:tests_linear} with constant velocities $a=b=1$ on the \emph{circular disk} 
\begin{equation}\label{eq:tests_disk}
	\Omega = \left\{ \, \mathbf{x} \in \R^2 \mid \norm{\mathbf{x} - \mathbf{c}} \leq r \, \right \}
\end{equation}
with center $\mathbf{c} = [1/2, 1/2]^T$ and radius $r = 1/2$. 
Also see \cref{fig:circle}. 
We compare the polynomial $\mathcal{F}_1$-exact and trigonometric $\mathcal{F}_2$-exact MFSBP operators already discussed in \cref{sub:examples_circle}. 
\cref{fig:SolCirc} illustrates the reference solution and the point-wise errors numerical solutions at time $t=1$ for the polynomial $\mathcal{F}_1$-exact and trigonometric $\mathcal{F}_2$-exact MFSBP operators. 
Once again, the trigonometric MSFBP operator yields a significantly more accurate numerical solution than the polynomial MSBP operator. 

\begin{figure}[tb]
	\centering 
	\begin{subfigure}{0.32\textwidth}
		\includegraphics[width=\textwidth]{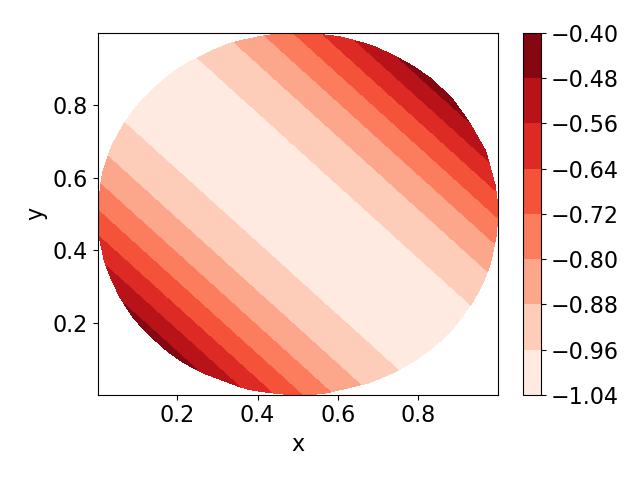}
		\caption{Reference solution}
		\label{fig:SolCirc_num_ref}
	\end{subfigure}
	\begin{subfigure}{0.32\textwidth}
		\includegraphics[width=\textwidth]{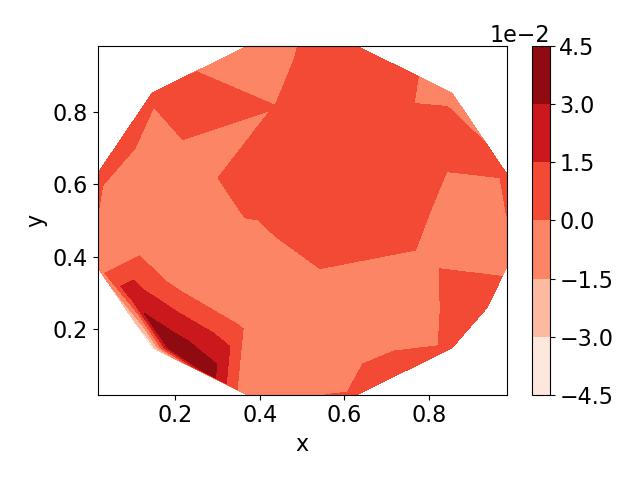}
		\caption{
		Error for the polynomial operator
		} 
		\label{fig:SolCirc_error_poly}
	\end{subfigure}
	\begin{subfigure}{0.32\textwidth}
		\includegraphics[width=\textwidth]{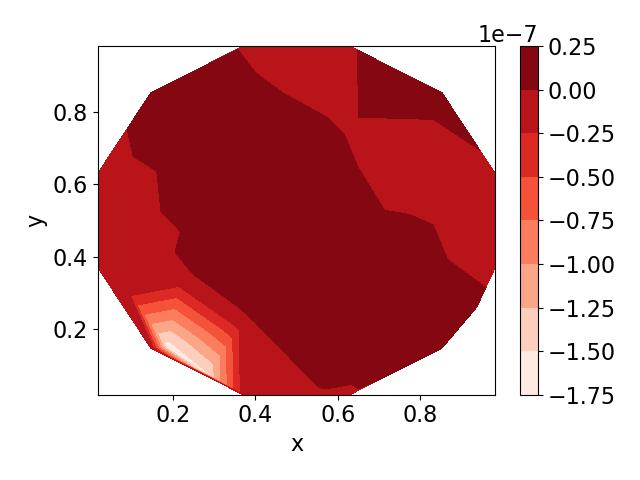}
		\caption{
		Error for the trigonometric operator
		} 
		\label{fig:SolCirc_error_trig}
	\end{subfigure}
	\caption{ 
	The reference solution and point-wise errors of the numerical solutions of the IBVP \cref{eq:testslinear} on the circular disk \cref{eq:tests_disk} at $t=1$ using the polynomial $\mathcal{F}_1$-exact and trigonometric $\mathcal{F}_2$-exact MFSBP operators. 
	The trigonometric $\mathcal{F}_2$-exact MFSBP operator yields significantly more accurate results. 
	}
	\label{fig:SolCirc}
\end{figure}

%% file: 9_summary.tex
\section{Concluding thoughts} 
\label{sec:summary} 

We introduced the concept of diagonal-norm MFSBP operators, thereby generalizing existing one-dimensional FSBP operators \cite{glaubitz2022summation} to multi-dimensional geometries and polynomial MSBP operators \cite{hicken2016multidimensional} to general function spaces. 
Our new MFSBP operators allow us to systematically develop stable and high-order accurate numerical methods that adapt to the characteristic behavior of the underlying solution. 
In this first paper, we focused on first establishing their theoretical foundation, including their mimetic properties and existence, and provided a general construction procedure. 
In particular, we showed that most mimetic properties of polynomial SBP operators carry over to the more general class of MFSBP operators. 
We stress that for MFSBP operators to mimic conservation, the associated approximation space needs to include constants, which polynomial MSBP operators satisfy by construction. 

Furthermore, we connected the existence of MFSBP operators and positive quadratures that are exact for certain, in general, non-polynomial function spaces. 
Building upon the theoretical existence investigation, we derived a general construction procedure for MFSBP operators. 
An essential part of this procedure is that a positive and $\partial_{\boldsymbol{\xi}} (\mathcal{F}^2)$-exact quadrature must exist if an $\mathcal{F}$-exact MFSBP operators approximating the directional derivative $\partial_{\boldsymbol{\xi}}$ is desired. 
While the existence and construction of such quadratures were established in \cite{glaubitz2022constructing} using a simple least squares approach, the resulting quadratures are not necessarily optimal in the sense that the number of quadrature points is notably larger than the dimension of the function space for which they are exact. 
Constructing optimal (so-called Gaussian) quadratures for general multi-dimensional function spaces is an open problem. 
We will further investigate the optimization and efficiency of our MFSBP operators. 
However, such efforts need to be tailored to specific function spaces and will be carried out in future work. 

While we demonstrated the advantage of using MFSBP operators to solve numerical PDEs for different linear problems, where prior knowledge of which approximation space should be used is readily available, a more exhaustive numerical study will be provided in future work. 
Future research efforts also include investigating the CFL limits and dispersion properties of MFSBP operators (similar to \cite{gassner2011comparison}), strategies for adaptively changing the approximation space (see \cite{yuan2006discontinuous}), 
stabilization techniques, such as split forms \cite{svard2014review,fernandez2014review,gassner2016split}, artificial dissipation \cite{mattsson2004stable,ranocha2018stability,offner2020analysis}, and other shock-capturing procedures, in combination with their application to discontinuous problems. 

%% file: app_POCS.tex
\section{The Projection Onto Convex Sets algorithm} 
\label{app:POCS} 

The Projection Onto Convex Sets (POCS) algorithm, sometimes also called the alternating projection algorithm, is a method to find a point in the intersection of two closed convex sets. 
Given the two overlapping convex sets $C$ and $D$, the POCS algorithm finds a point $x \in C \cap D$ by alternatingly projecting onto the sets $C$ and $D$. 
Given an arbitrary initial point $x_0$, the POCS algorithm produces a sequence of points $(x_k)_{k \in \N}$ with 
\begin{equation} 
	x_{k+1} = \operatorname{proj}_{D}\left( \operatorname{proj}_{C}(x_k) \right),
\end{equation} 
where $\operatorname{proj}_{C}$ and $\operatorname{proj}_{D}$ are projections onto $C$ and $D$, respectively. 
\cref{fig:POCS} illustrates the POCS algorithm for the case of $C$ and $D$ being (affine) linear spaces. 
See \cite{Neumann1950Functional,GUBIN1967The,escalante2011alternating} for more details on the POCS algorithm.

\subsection{The POCS algorithm and surface quadratures} 
\label{app:POCS_surface}

In \cref{sub:constr_B}, we pointed out the necessity of finding positive and $\mathcal{F}^2$-exact surface quadrature on each smooth boundary part $\Gamma_j$, $j=1,\dots,J$, to then construct a mimetic boundary operator. 
We can find such quadratures using the POCS algorithm. 
To this end, let $\{ f_l \}_{l=1}^L$ be a basis of $\mathcal{F}^2$. 
The exactness conditions \cref{eq:exactness_cond_surface} can then be formulated as a linear system for the weight vector $\mathbf{v} = [v_1,\dots,v_M]^T$, 
\begin{equation}\label{eq:POCS_surface_linear}
	A \mathbf{v} = \mathbf{b},
\end{equation}
where the coefficient matrix $A$ and the right-hand side vector $\mathbf{b}$ are given by 
\begin{equation}\label{eq:surfexact}
	A = 
	\begin{bmatrix} 
		( \boldsymbol{\xi} \cdot \boldsymbol{n}(\mathbf{x}^{(j)}_1) ) f_1(\mathbf{x}^{(j)}_1) & \dots & ( \boldsymbol{\xi} \cdot \boldsymbol{n}(\mathbf{x}^{(j)}_M) ) f_1(\mathbf{x}^{(j)}_M) \\ 
		\vdots & & \vdots \\ 
		( \boldsymbol{\xi} \cdot \boldsymbol{n}(\mathbf{x}^{(j)}_1) ) f_L(\mathbf{x}^{(j)}_1) & \dots & ( \boldsymbol{\xi} \cdot \boldsymbol{n}(\mathbf{x}^{(j)}_M) ) f_L(\mathbf{x}^{(j)}_M)
	\end{bmatrix}, 
	\quad 
	\mathbf{b} = 
	\begin{bmatrix} 
		I^{(\Gamma_j)}[f_1] \\ \vdots \\ I^{(\Gamma_j)}[f_L]
	\end{bmatrix}.
\end{equation} 
If $M > L$ and $A$ has linearly independent rows, then \cref{eq:POCS_surface_linear} has infinitely many solutions, which form an affine linear subspace of $\R^{M}$. 
In particular, the solution space is a closed convex set and can therefore be used as $C$ in the POCS algorithm, i.e., 
\begin{equation} 
	C = \left\{ \, \mathbf{v} \in \R^M \mid A \mathbf{v} = \mathbf{b} \, \right\}.
\end{equation}
The second restriction on the surface quadrature is positivity, i.e., $v_m > 0$ for all $m=1,\dots,M$. 
This motivates us to use 
\begin{equation} 
	D = \left\{ \, \mathbf{v} \in \R^M \mid v_m \geq 0,\, m=1,\dots,M \, \right\}
\end{equation}
as the second closed convex set in the POCS algorithm. 
We restrict the $v_m$'s to be non-negative instead of positive to ensure that $D$ is a \emph{closed} convex set.
While this does only ensure that the surface quadrature is non-negative, we can go over to a positive quadrature by removing all zero weights and the corresponding points. 
The orthogonal projections onto $C$ and $D$ are 
\begin{equation} 
	\label{eq:surfpocs}
\begin{aligned}
	\operatorname{proj}_{C}(\mathbf{v}) 
		& =  \mathbf{v} + A^+ \left( \mathbf{b} - A\mathbf{v} \right), \\ 
	\left[ \operatorname{proj}_{D}(\mathbf{v}) \right]_m
		& = \max(v_m, 0), \quad m=1,\dots,M,
\end{aligned}
\end{equation}
where $\left[ \operatorname{proj}_{D}(\mathbf{v}) \right]_m$ denotes the $m$-th component of $\operatorname{proj}_{D}(\mathbf{v}) \in \R^M$ and $A^+$ is the Moore--Penrose pseudo-inverse of $A$.
Since $A$ has linearly independent rows, $A A^T$ is invertible and $A^+ = A^T (A A^T)^{-1}$ constitutes a left inverse. 
See \cite{higham1988symmetric,escalante2011alternating} for more details.

\begin{remark}\label{rem:intersection_surface} 
	It was proved in \cite{glaubitz2021stableCFs,glaubitz2022constructing} that a positive and $\mathcal{F}^2$-exact least-squares quadrature is found whenever (i) $\mathcal{F}^2$ includes constants and (ii) sufficiently many points are used.	
	This implies that, under the same assumptions, the intersection of the closed and convex sets $C$ and $D$ is non-empty, and the POCS algorithm is ensured to converge to the desired weights of a positive and $\mathcal{F}^2$-exact surface quadrature. 
\end{remark}

\begin{remark}\label{rem:basis_surface} 
	If a basis of $\mathcal{F}^2$ is not readily available, the linear system \cref{eq:POCS_surface_linear}, encoding $\mathcal{F}^2$-exactness of the desired surface quadrature, can also be formulated using a generating set of $\mathcal{F}^2$. 
	If $\{ f_k \}_{k=1}^K$ is a basis of $\mathcal{F}$, then $\{ g_k \}_{k=1}^{K^2}$ with $g_{i + K(j-1)} = f_i f_j$, $i,j=1,\dots,K$, is a natural generating set. 
	The price we have to pay for the gained simplicity is that $\{ g_k \}_{k=1}^{K^2}$ contains more elements than the basis of $\mathcal{F}^2$, which makes the linear system \cref{eq:POCS_surface_linear} larger and therefore computationally more expensive to solve.  
	However, in our implementation, we observed the additional costs as insignificant compared to the overall computational costs. 
\end{remark}

\subsection{The POCS algorithm and volume quadratures} 
\label{app:POCS_volume}

As described in \cref{sub:constr_P}, we need to find a positive and $\partial_{\boldsymbol{\xi}} ( \mathcal{F}^2 )$-exact volume quadrature before we can construct an FSBP operator. 
To this end, we proceed similarly as in \cref{app:POCS_surface} and use the POCS algorithm. 
Given a basis $\{ f_l \}_{l=1}^L$ of $\partial_{\boldsymbol{\xi}} ( \mathcal{F}^2 )$, the exactness conditions \cref{eq:exactness_cond_volume} can again be formulated as the linear system 
\begin{equation}\label{eq:POCS_volume_linear}
	A \mathbf{w} = \mathbf{b}
\end{equation}
for the weight vector $\mathbf{w} = [w_1,\dots,w_N]^T$, where 
\begin{equation}\label{eq:volexact}
	A = 
	\begin{bmatrix} 
		f_1(\mathbf{x}_1) & \dots & f_1(\mathbf{x}_N) \\ 
		\vdots & & \vdots \\ 
		f_L(\mathbf{x}_1) & \dots & f_L(\mathbf{x}_N)
	\end{bmatrix}, 
	\quad 
	\mathbf{b} = 
	\begin{bmatrix} 
		I[f_1] \\ \vdots \\ I[f_L]
	\end{bmatrix}.
\end{equation} 
Like before, this yields the closed convex set 
\begin{equation} 
	C = \left\{ \, \mathbf{w} \in \R^N \mid A \mathbf{w} = \mathbf{b} \, \right\}.
\end{equation} 
The second restriction on the volume quadrature is positivity, i.e., $w_n > 0$ for all $n=1,\dots,N$. 
However, because of conditioning considerations, we further restrict the quadrature weights to lie above a certain threshold $w_{\mathrm{min}}$ and choose  
\begin{equation} 
	D = \left\{ \, \mathbf{w} \in \R^N \mid w_n \geq w_{\mathrm{min}},\, n=1,\dots,N \, \right\}.
\end{equation}
The motivation for introducing the threshold $w_{\mathrm{min}}$ is that it bounds the condition number of the inverse diagonal norm matrix by $\| P^{-1} \| \leq 1/w_{\mathrm{min}}$. 
Without such a threshold, a single small weight $w_n \approx 0$ could yield the inverse norm matrix to have an undesirable high condition number, which might translate into an ill-conditioned MFSBP operator $D_{\boldsymbol{\xi}} = P^{-1} Q_{\boldsymbol{\xi}}$. 
In our implementation, we choose $w_{\mathrm{min}} = \frac{1}{10 N}$, although we do not claim the optimality of this choice. 
The orthogonal projections onto $C$ and $D$ are given by 
\begin{equation}\label{eq:volpocs}
\begin{aligned}
	\operatorname{proj}_{C}(\mathbf{w}) 
		& =  \mathbf{w} + A^+ \left( \mathbf{b} - A\mathbf{w} \right), \\ 
	\left[ \operatorname{proj}_{D}(\mathbf{w}) \right]_n
		& = \max(w_n, w_{\mathrm{min}} ), \quad n=1,\dots,N,
\end{aligned}
\end{equation}

Similar arguments as in \cref{rem:intersection_surface} ensure that the intersection of $C$ and $D$ is non-empty and that the POCS algorithm converges to a positive and $\partial_{\boldsymbol{\xi}} ( \mathcal{F}^2 )$-exact volume quadrature.

\begin{remark}\label{rem:basis_volume} 
	Following up on the discussion in \cref{rem:basis_surface}, the linear system \cref{eq:POCS_volume_linear} can also be formulated using a generating set of $\partial_{\boldsymbol{\xi}} ( \mathcal{F}^2 )$, if no basis is readily available. 
	If $\{ f_k \}_{k=1}^K$ is a basis of $\mathcal{F}$, then $\{ g_k \}_{k=1}^{K^2}$ with $g_{i + K(j-1)} = \partial_{\boldsymbol{\xi}} (f_i f_j)$, $i,j=1,\dots,K$, spans $\partial_{\boldsymbol{\xi}} ( \mathcal{F}^2 )$. 
\end{remark}

\subsection{Using POCS for constructing $Q_A$} 
\label{app:POCS_QA}

We now comment on some computational details for constructing the anti-symmetric matrix $Q_A$ discussed in \cref{sub:constr_Q}, which has to satisfy the exactness condition \cref{eq:QA_exactness}. 
We can determine such a matrix $Q_A$ using the POCS algorithm. 
To this end, note that the exactness condition \cref{eq:QA_exactness} and the restriction to anti-symmetric $Q_A$'s ($Q_A^T = - Q_A$) define nonempty closed convex sets,
\begin{equation}\label{eq:POCS_sets}
\begin{aligned}
	C & = \left\{ \, Q_{A} \in \R^{N \times N} \mid Q_{A} V = P V_{\boldsymbol{\xi}} - \frac{1}{2} B_{\boldsymbol{\xi}} V \, \right\}, \\ 
	D & = \left\{ \, Q_{A} \in \R^{N \times N} \mid Q_{A} + Q_{A}^T = 0 \, \right\}, 
\end{aligned}
\end{equation}
in the space of $N \times N$ matrices. 
If the intersection between the two convex sets $C$ and $D$ is non-empty, then the POCS algorithm finds a point in the intersection by alternatingly projecting onto $C$ and $D$. 
The orthogonal projections onto the sets in \cref{eq:POCS_sets} are given by 
\begin{equation}\label{eq:POCS_proj}
\begin{aligned}
	\operatorname{proj}_{C}(Q_{A}) 
		& = Q_{A} + \left(P V_{\boldsymbol{\xi}} - \frac{1}{2} B_{\boldsymbol{\xi}} V - Q_{A} V \right) V^+ , \\ 
	\operatorname{proj}_{D}(Q_{A}) 
		& = \left( Q_{A} - Q_{A}^T \right)/2, 
\end{aligned}
\end{equation}
respectively. Here, $V^+$ is the Moore--Penrose pseudo-inverse of $V$.
Since $V$ has linearly independent columns, $V^T V$ is invertible and $V^+ = (V^T V)^{-1} V^T$ constitutes a right inverse. 
See \cite{higham1988symmetric,escalante2011alternating} for more details.

\subsection{Robust implementation} 

In some cases, the coefficient matrix $A$ in the exactness conditions $A \mathbf{w} = \mathbf{b}$, see \cref{eq:POCS_surface_linear,eq:POCS_volume_linear}, was observed to be ill-conditioned. 
We formulated and solved the exactness conditions for a reduced discrete orthonormal basis in such cases. 
Starting from a potentially over-complete spanning set $\{ g_l \}_{l=1}^L$ of a function space $\mathcal{G}$ ($\mathcal{G} = \mathcal{F}^2$ or $\mathcal{G} = \partial_{\boldsymbol{\xi}}(\mathcal{F}^2)$), we will now describe an algorithm that computes an orthonormal basis spanning approximately the same space. 
We start by computing the coefficient matrix $A$ in the exactness conditions \cref{eq:POCS_surface_linear,eq:POCS_volume_linear} using the original spanning set $\{ g_l \}_{l=1}^L$. 
We then determine the singular value decomposition (SVD) of $A$. 
Recall that the SVD of $A \in \R^{L \times N}$ with $L \leq N$ is given by 
\begin{equation}\label{eq:SVD}
	A = U \Sigma V^T,
\end{equation} 
where $U \in \R^{L \times L}$ and $V \in \R^{N \times N}$ are orthogonal matrices and $\Sigma \in \R^{L \times N}$ is a rectangular diagonal matrix with non-negative singular values $\sigma_1 \geq \dots \geq \sigma_L$ on the diagonal. 
Let $\varepsilon > 0$ be a fixed threshold that remains to be determined. 
We can now ``compress" the original set $\{ g_l \}_{l=1}^L$ in the sense that a sizeable rank-deficient collection of these functions is replaced with an orthonormal basis $\{ h_m \}_{m=1}^M$ whose dimension is the numerical rank of the spanning set up to precision $\varepsilon$. 
The basis is orthonormal in the sense that $\mathbf{h_m}^T \mathbf{h_n}$ is equal to one if $m=n$ and zero otherwise. 
We get the orthonormal basis functions as 
\begin{equation}
	h_m(x) = \frac{1}{\sigma_m} \sum_{l=1}^N u_{l,m} g_l(x), \quad 
	m=1,\dots,M,
\end{equation} 
where $M$ is the largest integer such that $\sigma_M \geq \varepsilon$. 
If $A$ was invertible and $\epsilon = 0$, this construction is equivalent to multiplying both sides of the SVD $A = U \Sigma V^T$ from the left-hand side with $\Sigma^{-1} U^T$, yields 
\begin{equation}
	V^T = \Sigma^{-1} U^T A. 
\end{equation}
The functions $h_1,\dots,h_M$ constitute an orthonormal basis for the span of the input functions $g_1,\dots,g_L$ to precision $\varepsilon$. 
Since the basis is orthonormal, the condition number of the resulting coefficient matrix $A_{\varepsilon}$ is equal to one. 
The new basis can be used to calculate a new corresponding right side $\mathbf{b}_{\varepsilon}$ and therefore allows us to represent the exactness constraints in a stable manner. 
A crucial part is selecting the threshold $\varepsilon$. 
Too large values discard too many dimensions and lead to a matrix $A_{\varepsilon}$ that describes fewer exactness conditions than the original one. 
At the same time, too small values inflate the norm of $b_{\varepsilon}$. 
In our implementation, we solve $A_{\varepsilon} \mathbf{w} = \mathbf{b}_{\varepsilon}$ for $\varepsilon \in \{10^{-16},\dots,10^{-5}\}$ and ultimately select the resulting weight vector $\mathbf{w}_{\varepsilon}$ that minimizes $\|A \mathbf{w}_{\varepsilon} - \mathbf{b} \|_2$, i.e., the weights that best-fit the original exactness conditions.